\documentclass[11pt,a4paper,oneside]{amsart}
\usepackage[centertags]{amsmath}
\usepackage{bbm,amsfonts,amssymb, amssymb, amsthm}
\usepackage{indentfirst}
\usepackage{mathrsfs}
\usepackage{cancel}
\usepackage{newlfont}
\usepackage{comment}
\usepackage[all]{xy}

\usepackage{float}
\usepackage{color}
\newlength{\defbaselineskip}
\usepackage{booktabs}
\usepackage{tikz-cd}

\usepackage{stackengine}
\newcommand\mathcalbb[2][2]{%
  \stackengine{0pt}{$\mathcal{#2}$}{$\mkern#1mu\mathcal{#2}$}{O}{l}{F}{F}{L}}
\newtheorem{thm}{Theorem}[section]
\newtheorem{lem}[thm]{Lemma}
\newtheorem{prop}[thm]{Proposition}
\newtheorem{cor}[thm]{Corollary}

\newtheorem{rmk}[thm]{Remark}
\newtheorem{ex}[thm]{Example}
\newtheorem{defn}[thm]{Definition}

\numberwithin{equation}{section}
\newcommand{\Cour}[1]      {\left[\!\left[#1\right]\!\right]}

\newcommand{\llbracket}{\langle\!\langle}
\newcommand{\rrbracket}{\rangle\!\rangle}
\newcommand{\<}{\langle}
\renewcommand{\>}{\rangle}

\newcommand{\Lie}       {\mathcal{L}}
\newcommand{\frakx}     {\mathfrak{X}}

\newcommand{\N}     {\mathcal{N}} %Nijenhuis torsion
\newcommand{\an}    {\mathfrak{a}}
\newcommand{\C}     {\mathbb{C}}   
\newcommand{\R}     {\mathbb{R}}
\newcommand{\T}     {\mathbb{T}}
\newcommand{\E}     {\Gamma_\mathcal{E}}

\newcommand{\D}     {\mathbb{D}}

\newcommand{\pr}    {\mathrm{pr}}
\newcommand{\ii}    {\mathbf{i}\,}

\renewcommand{\O}     {\mathcal{O}}

\newcommand{\Real}   {\mathrm{Re}}

\newcommand{\hol}   {\mathrm{hol}}
\newcommand{\dob}    {\overline{\partial}}
\newcommand{\Panchor} {\mathfrak{p}}
\newcommand{\newcourant} {C^{\infty}(U, \C) \otimes_{\mathcal{O}(U)} \mathcal{E}(U)}
\renewcommand{\Re}      {\mathrm{Re}}
\renewcommand{\Im}      {\mathrm{Im}}
\newcommand{\Cinfty}     {{\scriptscriptstyle{C^\infty}}}
\newcommand{\calD} {\mathcal{D}}

\usepackage[colorinlistoftodos]{todonotes}
\title{Courant-Nijenhuis algebroids}

\author{Henrique Bursztyn} 
\address{IMPA, Estrada Dona Castorina 110, Rio de Janeiro, 22460-320, Brazil.}
\email{henrique@impa.br}
\author{Thiago Drummond}
\address{Departamento de Matem\'atica, Instituto de Matem\'atica,
Universidade Federal do Rio de Janeiro,
Caixa Postal 68530, Rio de Janeiro, RJ, 21941-909, Brasil.}
\email{drummond@im.ufrj.br}
\author{Clarice Netto}
\address{Instituto de Matem\'atica e Estat\'{i}stica, Universidade de S\~ao Paulo, Rua do Mat\~ao 1010, Cidade Universit\'aria, 05508-090 S\~ao Paulo, Brasil.}
\email{cnetto@ime.usp.br}

\date{}

\begin{document}
\maketitle

\begin{abstract} 
    We introduce Courant 1-derivations, which describe a compatibility between Cou\-rant algebroids and linear (1,1)-tensor fields and lead to the notion of Courant-Nijenhuis algebroids. We provide examples of Courant 1-derivations on exact Courant algebroids and show that holomorphic Courant algebroids can be viewed as special types of Courant-Nijenhuis algebroids. By considering Dirac structures, one 
    recovers the Dirac-Nijenhuis structures of \cite{BDN} (in the special case of the standard Courant algebroid) and  obtains an equivalent description of Lie-Nijenhuis bialgebroids \cite{Dru} via Manin triples.
 \end{abstract}

\tableofcontents

%\begin{itemize}
 %   \item Definition of Courant-Nijenhuis algebroid on general context. Characterization of CN structures using Courant morphisms (not proved yet). 
 %   \item Examples
 %   \begin{itemize}
 %       \item $\mathbb{T}M$: Characterization of Courant 1-derivations using $(\mathbb{D}^r,(r,r^*),r)$. 
 %       Other examples on $\mathbb{T}M$ coming from metrics.
 %       \item holomorphic Courant algebroids
 %   \end{itemize}
 %   \item Manin triples for CN structures
%\end{itemize}

\section{Introduction}

A 1-derivation on a vector bundle $E\to M$ is a connection-like object that codifies a linear (1,1)-tensor field on the total space of $E$ in the same way that usual derivations of vector bundles correspond to linear vector fields. 
As it turns out, many examples of compatibility conditions involving structures of interest in Poisson geometry can be conveniently expressed in terms of such 1-derivations; in this context, a special role is played by ``Nijenhuis 1-derivations'', i.e., 1-derivations whose corresponding linear (1,1)-tensor fields have vanishing Nijenhuis torsion.

A motivating example is that of Poisson-Nijenhuis structures \cite{MM}, originally formulated in terms of 
an intricate notion of compatibility involving Poisson structures and  Nijenhuis operators (that includes the vanishing of the so-called Magri-Morosi concomitant). 
From the recent viewpoint of  \cite{HT,Dru}, this is understood as follows: a Nijenhuis operator $r$ on a manifold $M$ canonically gives rise to a 1-derivation on $TM$ and a dual one on $T^*M$, and the compatibility of $r$ with a Poisson bivector field $\pi$ is simply that $\pi^\sharp: TM\to T^*M$ intertwines these 1-derivations. This perspective leads the way to different generalizations, such as
\begin{itemize}
\item[(a)] Dirac-Nijenhuis structures \cite{BDN},

%\item[(b)] Lie-Nijenhuis algebroids \cite{HT},

\item[(b)] Lie-Nijenhuis bialgebroids \cite{Dru}.
\end{itemize}

Just as Poisson structures are particular examples of both Dirac structures and Lie bialgebroids \cite{Cour,LiuXuWe},
Poisson-Nijenhuis structures are special cases of the objects in (a) and (b). In this paper, we take a step further and introduce
%the compatibility of 1-derivations with Courant algebroids.
the closely related notion of {\em Courant-Nijenhuis algebroid}.

Dirac-Nijenhuis structures and Lie-Nijenhuis bialgebroids are motivated by the theory of Lie groupoids, in that they arise as infinitesimal counterparts of presymplectic-Nijenhuis and Poisson-Nijenhuis group\-oids, respectively, generalizing the correspondence between Poisson-Nijenhuis structures and symplectic-Nijenhuis groupoids from \cite{SX}, see \cite{BDN,Dru}. 
An important class of examples is given by holomorphic structures. Any holomorphic vector bundle can be seen as a smooth real vector bundle equipped with a special type of Nijenhuis 1-derivation, that we call a ``Dolbeault 1-derivation''. 
In this particular context, (a) and (b) recover holomorphic Dirac structures and holomorphic Lie bialgebroids, respectively, and their integrations correspond to holomorphic presymplectic and holomorphic Poisson groupoids.

%**** One motivation to consider the objects above comes from the theory of Lie groupoids. A central fact in Poisson geometry is that Poisson structures are infinitesimal counterparts of symplectic groupoids. Poisson-Nijenhuis correspond to symplectic-Nijenhuis groupoids ****cite****. The objects above allow this result to be extended in different directions: Dirac-Nijenhuis to presymplectic-Nijenhuis groupoids, Lie Nijenhuis to Lie groupoids with multiplicative Nijenhuis structure and Lie Nijenhuis bialgebroids to Poisson-Nijenhuis groupoids.

Courant algebroids are central ingredients in the theory of Dirac structures and Lie bialgebroids.
The main object of study in this paper is a notion of compatibility between 1-derivations and Courant algebroids, described in Definition \ref{def:Cder} by what we call a {\em Courant 1-derivation}.
%The corresponding Dirac structures are taken to be invariant by the 1-derivation (Definition~\ref{defDN}). 
A Courant algebroid equipped with a Courant 1-derivation that is also Nijenhuis is called a {\em Courant-Nijenhuis algebroid}. Dirac structures therein generalize the Dirac-Nijenhuis structures of \cite{BDN} and provide an alternative approach to the Lie-Nijenhuis bialgebroids of \cite{Dru} via Manin triples.

The paper is structured as follows. In $\S$ \ref{sec:1der} we recall 1-derivations on vector bundles, their main examples and properties, including the notion of duality, the Nijenhuis condition, and their compatibility with (pre-)Lie algebroid structures. 
Courant 1-derivations and Courant-Nijenhuis algebroids are introduced in $\S$ \ref{sec:1derCA}, along with their Dirac structures.
As a basic example, we show that any (1,1)-tensor field on a manifold $M$ canonically defines a Courant 1-derivation on the standard Courant algebroid $\T M=TM \oplus T^*M$ that underlies the Dirac-Nijenhuis structures studied in \cite{BDN}. 
In $\S$ \ref{sec:T+T*} we consider more general Courant 1-derivations on $\T M$, including a class of examples obtained through  the additional choice of a (pseudo-)riemannian metric; in this case, the Nijenhuis condition is shown to be related to the K\"ahler compatibility condition (Prop.~\ref{prop:Nij}).
In $\S$ \ref{sec:hol} we show (Theorem~\ref{thm:holcourant}) that Courant-Nijenhuis algebroids defined by Dolbeault 1-derivations coincide with 
 holomorphic Courant algebroids.
In $\S$ \ref{sec:double} we consider lagrangian splittings of Courant algebroids equipped with Courant 1-derivations (Theorem~\ref{thm:manin}). We show in particular that the Drinfeld double of a Lie-Nijenhuis bialgebroid is a Courant-Nijenhuis algebroid, thereby establishing an equivalence between Lie-Nijenhuis bialgebroids and Courant-Nijenhuis algebroids equipped with splittings by Dirac-Nijenhuis structures.

%what is it about? Compatibility of 1-derivations with Courant structures. special focus on nijenhuis. initial motivation to extend PN to dirac  and bialgebroids (integration problems, two ways to extend PN to symplectic-N groupoids)...

%- compatibility with Lie algebroids and bialgebroids, give infinitesimal counterparts of multiplicative Nijenhuis, Poisson-Nijenhuis.... all understood as dirac structures in a nijenhuis courant algebroid. 

%- compatibility with Dirac in T+T*, pre-symplectic Nijenhuis groupoids...

%The Courant compatibility equations originally appeared in \cite[Lem.~6.1]{BDN} as a key property of $\D^r$, used in the integration of  Dirac-Nijenhuis structures. Our starting point in this work is the realization that these compatibility equations make sense for arbitrary Courant algebroids.

%A particular 1-derivation is holomorphic structure. In this case compatibility is holomorphic Courant algebroids.

\begin{rmk}{\em
Other works in the literature involve Nijenhuis structures and Courant algebroids but follow a different direction. Nijenhuis tensors on Courant algebroids have been considered by various authors \cite{CarGraMar,grab,Yvette}, with applications to deformations, hierarchies and compatibilities of geometric structures, see e.g. \cite{ACXNC,ANC}.  The fundamental objects of interest in these papers are vector bundle endomorphisms $N: E\to E$ with vanishing Nijenhuis torsion, where $E$ is a Courant algebroid and Nijenhuis torsion is defined with respect to the Courant bracket. In contrast, the present paper studies a notion of compatibility between a (ordinary) linear Nijenhuis operator $K: TE\to TE$ on the total space of $E$ and the Courant structure on $E$.}
\end{rmk}

%\textcolor{blue}{Other approaches of Courant algebroids that hold a Nijenhuis operator are present in the literature. For example, in \cite{ACXNC,ANC,CarGraMar,grab} and \cite{Yvette} the authors explored others notions of a Courant algebroid compatible with a Nijenhuis operator. These approaches resulted in a different point of view from ours, more related to the supergeometry theory, and involves a notion of ''deformation'' of a Courant algebroid by a Nijenhuis operator. Our focus was introduce a notion of Courat-Nijenhuis algebroids that is related to holomorphic Courant algebroids in the sense of \cite{BDN, SX}.}

\medskip

\noindent {\bf Acknowledgments}. This project was partially supported by CNPq (National Council for Scientific and Technological Development), FAPERJ (Rio de Janeiro State Research Foundation) and by grant $\sharp$ 2022/06205-2 of FAPESP (São Paulo State Research Foundation).

%%%%%%%%%%%%%%%%%%%%%%%%%%%%%%%%%%%%%%%%%%%%%%%%
\section{1-Derivations on vector bundles}\label{sec:1der}
A central role in this paper is played by 1-derivations on vector bundles,
so we start by briefly recalling their definition and basic properties.

Let $p:E \to M$ be a smooth, real vector bundle. A {\em 1-derivation} on $E\to M$ is a triple $\mathcal{D} = (D,l,r)$, where $r: TM \to TM$ and $l:E \to E$ are vector-bundle maps covering the identity, and $D: \Gamma(E) \to \Gamma(T^*M\otimes E)$ is an $\R$-linear map satisfying the following Leibniz-type condition:
\begin{equation}\label{eq:Leibniz}
D_X(f\sigma) = fD_X(\sigma) + (\Lie_Xf) \, l(\sigma) - (\Lie_{r(X)}f) \,\sigma,
\end{equation}
where $X\in \mathfrak{X}(M)$, $\sigma\in \Gamma(E)$ and $f\in C^\infty(M)$, and we use the notation 
$$
D_X: \Gamma(E)\to \Gamma(E), \qquad D_X(\sigma) = i_X(D(\sigma)).
$$
Note that the linear combination of 1-derivations (defined componentwise) is a 1-derivation.

Let $F\subseteq E$ be a subbundle (over the same base, for simplicity).

\begin{defn}\label{def:compat}
We say that $F$ is {\em $\mathcal D$-invariant} if 
\begin{itemize}
    \item $l(F)\subseteq F$,
    \item $D_X (\Gamma(F))\subseteq \Gamma(F), \qquad \forall X\in \mathfrak{X}(M)$.
\end{itemize}
%\end{defn}
\end{defn}

When $F$ is $\mathcal D$-invariant, $\mathcal{D}$ restricts to a 1-derivation on $F$, 
$$
(D|_{\Gamma(F)}, l|_F, r).
$$

%1-Derivations \textcolor{blue}{or 1-derivations?} are higher-degree analogs of ordinary derivations on vector bundles in a sense that we now explain.

%%%%%%%%%%%%%%%%%%%%%%%%%%%%%%%%%%%%%%%%%%%%%%%%%%%%%%%%%
\subsection{Equivalence with linear (1,1)-tensors fields}\label{subsec:1,1}
%Let $M$ be a smooth manifold and $p:E \to M$ a vector bundle. 

Just as usual derivations on a vector bundle $E\to M$ are equivalent to linear vector fields on $E$ (see e.g. \cite{Mac-book}), 1-derivations are in bijective correspondence with linear (1,1)-tensor fields, see \cite[$\S$~6.1]{HT} and \cite[$\S$~2.1]{BDN}.

Let $K \in \Omega^1(E,TE)$, i.e. $K$ is a (1,1)-tensor field on the total space $E$, that we view as a 
vector-bundle morphism $K:TE \to TE$. We say that $K$ is {\em linear} if $K:TE \to TE$ is also a vector-bundle morphism from the tangent prolongation bundle $TE \to TM$ to itself (not necessarily covering the identity on $TM$). Linear (1,1)-tensor fields form a linear subspace of $\Omega^1(E,TE)$.
The reader can find a detailed treatment of linear tensors in \cite{HT}, where the results stated in this section can be found.

Any linear (1,1)-tensor field $K: TE \to TE$ gives rise to a 1-derivation $\mathcal{
D}=(D, l, r)$ on $E$ as follows: $r: TM \to TM$ is the restriction of $K|_M$ to vectors tangent to the zero section, $l: E \to E$ is the restriction of $K|_M$ to vectors tangent to the $p$-fibers, and $D: \Gamma(E) \to \Gamma(T^*M \otimes E)$ is the $\R$-linear map defined by
$$
D_X(\sigma) = (\Lie_{\sigma^\uparrow}K)(X), \,\qquad \sigma \in \Gamma(E), \, X \in TM,
$$
where $\sigma^\uparrow \in \frakx(E)$ is the {\em vertical lift} of $\sigma \in \Gamma(E)$, given by $\sigma^\uparrow(e) = \left.\frac{d}{dt}\right|_{t=0} e+t \,\sigma(p(e))$. This assignment establishes a linear bijection between linear (1,1)-tensor fields and 1-derivations on $E$ with natural functorial properties, see \cite[Thm.~2.1]{BDN}. A direct consequence is that a subbundle $(F\to M) \subset (E\to M)$ is $\mathcal{
D}$-invariant if and only if it is preserved by the corresponding linear (1,1) tensor field $K$,
$$
K(TF)\subseteq TF.
$$

%\textcolor{blue}{use notation $\Omega^1(M,TM)$, $\Omega^1_{lin}(M,TM)$...?}

%A linear $(1,1)$-tensor is completely characterized by its first jet along the zero section of $E$. More precisely, there is a 1-1 correspondence between linear (1,1)-tensors $K: TE \to TE$ and triples $(D,l,r)$, where $r: TM \to TM$ is the restriction of $K|_M$ to vectors tangent to the zero section, $l: E \to E$ is the restriction of $K|_M$ to vectors tangent to the $p$-fibers and $D: \Gamma(E) \to \Gamma(T^*M \otimes E)$ is the $\R$-linear map defined by
%$$
%D_X(\sigma) = (\Lie_{\sigma^\uparrow}K)(X), \,\, \sigma \in \Gamma(E), \, X \in TM.
%$$
%Here $\sigma^\uparrow \in \frakx(E)$ is the vector field known as the vertical lift of $\sigma \in \Gamma(E)$ and characterized by $\sigma^\uparrow(e) = \left.\frac{d}{dt}\right|_{t=0} e+t \,\sigma(p(e))$. Note that $D$ satisfies a Leibniz-type equation
%\begin{equation}\label{eq:Leibniz}
%D_X(f\sigma) = fD_X(\sigma) + (\Lie_Xf) \, l(\sigma) - (\Lie_{r(X)}f) \,\sigma.
%\end{equation}

%Triples $(D,l,r)$, where $r: TM \to TM$ and $l:E \to E$ are vector bundle maps covering the identity and $D: \Gamma(E) \to \Gamma(T^*M\otimes E)$ is a $\R$-linear map satisfying \eqref{eq:Leibniz}, are known as \textit{1-derivations}. The correspondence between linear (1,1)-tensors and 1-derivations extends the well-known correspondence between linear vector fields and derivations on vector bundles \cite{Mac-book}.

Let us give some examples.

\begin{ex}[Connections]\label{ex:connection}\em
A connection $\nabla$ on $E$ defines a 1-derivation with $r=0$, $l=\mathrm{id}_E$ and $D=\nabla$. The corresponding linear (1,1)-tensor field $K: TE \to TE$ is the projection operator on the vertical bundle $\mathrm{ker}(Tp)\subseteq TE$ along the horizontal bundle defined by $\nabla$. \hfill $\diamond$
\end{ex}

\begin{ex}[Tangent and cotangent lifts] \label{ex:lifts}\em Given a (1,1)-tensor  on $M$, $r: TM \to TM$, its tangent lift $r^{tg}: T(TM) \to T(TM)$ and cotangent lift $r^{ctg}: T(T^*M) \to T(T^*M)$ are linear (1,1)-tensor fields on $TM$ and $T^*M$, respectively, defined as follows:
$$
r^{tg} = \Theta \circ Tr \circ \Theta, \qquad (\omega_{can})^{\flat}\circ r^{ctg} = (\varphi_r^*\,\omega_{can})^{\flat},
$$
where $\Theta: T(TM) \to T(TM)$ is the canonical involution of the double tangent bundle $T(TM)$, $\omega_{can}$ is the canonical symplectic form on $T^*M$, and $\varphi_r: T^*M \to T^*M$ is just $r^*$ seen as a smooth map from $T^*M$ to itself. The linearity of  $r^{tg}$ and $r^{ctg}$  was proved in \cite[Thm.~3.4]{Dru}, along with the  identification of their corresponding 1-derivations as 
$$
\mathcal{D}^r=(D^r, r, r), \quad\;\; \mathrm{ and } \quad \;\; \mathcal{D}^{r,*}=(D^{r,*},r^*,r),
$$ 
where, for $X\in \mathfrak{X}(M)$, $D^r_X: \mathfrak{X}(M)\to \mathfrak{X}(M)$ and $D^{r,*}_X: \Omega^1(M) \to \Omega^1(M)$ are given by
\begin{align}
\label{defn:Dr} D^r_X(Y) & = (\Lie_Yr)(X) = [Y,r(X)]-r([Y,X]),\\
\label{defn:Dr_dual} D^{r,*}_X(\alpha) & = \Lie_{X}(r^*\alpha) - \Lie_{r(X)}\alpha.
\end{align}
 \hfill $\diamond$
\end{ex}

\begin{ex}[Holomorphic structures and Dolbeault 1-derivations]\label{ex:hol}\em Let $r:TM \to TM$ be a complex structure on $M$.  A holomorphic vector bundle $\mathcal{E} \to M$ can be regarded as a real vector bundle $E\to M$ equipped with a fibrewise complex structure $l\in \mathrm{End}(E)$, $l^2=-\mathrm{Id}$, and a flat  $T^{0,1}$-connection $\dob$
on the complex vector bundle $(E,l)$ \cite{rawnsley}. 
%We refer to the triple $(r, l, \dob)$ as a {\em holomorphic structure} on $E\to M$. 
One can equivalently express the holomorphic structure on $E$ determined by $r$, $l$ and $\overline{\partial}$ as a  1-derivation $\mathcal{D}^{Dolb} = (D, l, r)$, where
\begin{equation}\label{eq:Dhol}
D_X(\sigma) = l(\dob_{X+\mathbf{i}r(X)}\sigma).
\end{equation}
Such 1-derivations arising from holomorphic structures will be referred to as {\em Dolbeault 1-derivations}, and they are characterized by the fact that the corresponding linear (1,1)-tensor fields $K: TE \to TE$ are complex structures  on the total space of $E$ (see Example~\ref{ex:holrev} below). Subbundles of $E$ which are $\mathcal{D}^{Dolb}$-invariant (in the sense of Def.~\ref{def:compat}) are holomorphic subbundles.

As special cases, the holomorphic structures on $TM$ and $T^*M$ induced by a complex structure $r: TM\to TM$ are given by the 1-derivations $\mathcal{D}^r$ and $\mathcal{D}^{r,*}$  from the previous example, with corresponding linear complex structures $r^{tg}$ and $r^{ctg}$, see  \cite[$\S$ 5.1]{Dru}.  \hfill $\diamond$

%equipped with a \textit{holomorphic structure}, i.e., a pair $(l, \nabla)$ where $l\in \mathrm{End}(E)$ satisfies $l^2=-\mathrm{Id}$  and $\nabla$ \textcolor{blue}{how about notation $\overline{\partial}$? terminology holomorphic 1-derivations?} is a flat $T^{0,1}$-connection on the complex vector bundle $(E,l)$. So we refer to the triple *** as a hol structure ****One can equivalently express the holomorphic structure on $E$ as a  1-derivation $(D^{hol}, l, r)$, where
%\begin{equation}\label{eq:Dhol}
%D^{hol}_X(\sigma) = l(\nabla_{X+\mathbf{i}r(X))}\sigma).
%\end{equation}
%characterized by its \textit{holomorphic structure}, i.e. a pair $(\nabla, l)$, where $\nabla$ is a flat $T^{0,1}$-connection on $E$ and $l:E \to E$ is the fiberwise multiplication by $\mathbf{i}$. The holomorphic structure is equivalent to the 1-derivation $(D^{hol}, l, r)$, where
%$$
%D^{hol}_X(\sigma) = l(\nabla_{X+\mathbf{i}r(X))}\sigma).
%$$
%Such 1-derivations arising from holomorphic structures are characterized by the fact that the corresponding linear (1,1)-tensor fields $K: TE \to TE$ are complex structures  on the total space of $E$. 

%$r^{tg}$ and $r^{ctg}$ are the holomorphic structures on the holomorphic tangent bundle $TM$ and cotangent bundle $T^*M$, respectively. In particular, $(D^r, r,r)$ and $(D^{r,*}, r^*, r)$ are the 1-derivations associated to the holomorphic structures. We refer to \cite{HT, Dru} for more details.
\end{ex}

%%%%%%%%%%%%%%%%%%%%%%%%%%%%%%%%%%%%%%%%%%%%%%%%%%%%%%%%%%%%%%%%%%%

\subsection{The Nijenhuis condition}

Let $T:TN\to TN$ be a (1,1)-tensor field on a manifold $N$. 
The Nijenhuis torsion of $T$ is $\N_T\in \Omega^2(E,TE)$ given by
$$
\N_T(X_1,X_2) = [T(X_1),T(X_2)]-T([X_1,X_2]_T), \,\,\, X_1,\,X_2 \in \frakx(N),
$$
where $[X_1,X_2]_T= [T(X_1),X_2]+[X_1,T(X_2)]-T([X_1,X_2]$ is the deformed bracket. If $\N_T=0$, $T$ is called a {\em Nijenhuis operator}.

%We illustrate this idea in two cases of relevance in this paper.

By the equivalence between 1-derivations and linear (1,1)-tensor fields, properties of the latter can be expressed in terms of the former. 
Let $\mathcal{D} = (D, l, r)$ be a 1-derivation with corresponding
linear (1,1)-tensor field $K$ on $E$.
%A central condition in this paper is the vanishing of the Nijenhuis torsion of $K$:
%A key idea which we pursue in this work is that 1-derivations provide a very fruitful perspective  to study properties of the corresponding linear (1,1)-tensor field. Nice illustrations of this idea are the sets of equations corresponding to the algebraic condition $K^2 = -\mathrm{id}_{TE}$ and to the vanishing of the Nijenhuis torsion of $K$:
%$$
%\N_K(U_1,U_2) = [K(U_1),K(U_2)]-K([U_1,U_2]_K), \,\,\, U_1,\,U_2 \in \frakx(E),
%$$
%where $[U_1,U_2]_K= [K(U_1),U_2]+[U_1,K(U_2)]-K([U_1,U_2]$ is the deformed bracket. If this holds $K$ is called a {\em Nijenhuis operator}. 
It is proven in \cite[$\S$~6]{HT} that 
\begin{equation}\label{eq:van_Nijenhuis}
 \N_K= 0 \Longleftrightarrow \left\{
 \begin{array}{l}
 \N_r=0, \\
 D_X(l(\sigma))-l(D_X(\sigma))=0,\\
 l(D_{[X, Y]}(\sigma)) - [D_X, D_Y](\sigma) - D_{[X,Y]_r}(\sigma) = 0,
 \end{array}
 \right.
\end{equation}
where $[D_X, D_Y]$ is the commutator of operators on $\Gamma(E)$. 
We refer to the equations on the right hand side as the \textit{Nijenhuis equations} for the 1-derivation $(D,l,r)$, and
 1-derivations satisfying them will be called {\em Nijenhuis 1-derivations}.

 We will also consider the algebraic condition $K^2 = -\mathrm{id}_{TE}$, saying that $K$ is an almost complex structure on the total space of $E$. In this case, it is also proven in \cite[Cor.~6.2]{HT} that
\begin{equation}\label{eq:square_minus_id}
 K^2= -\mathrm{id}_{TE} \Longleftrightarrow \left\{
 \begin{array}{l}
 r^{2}=-\mathrm{id}_{TM},\\  
 l^2=-\mathrm{id}_{E},\\
 D_{r(X)}(\sigma) + l(D_X(\sigma))=0.
 \end{array}
 \right.
\end{equation}
%We refer to the equations on the right hand side of \eqref{eq:square_minus_id} and \eqref{eq:van_Nijenhuis} as the \textit{almost complex equations} and \textit{Nijenhuis equations} for the 1-derivation $(D,l,r)$, respectively.
%We will call a 1-derivation satisfying the Nijenhuis equations as a {\em Nijenhuis 1-derivation}.

\begin{ex}[Dolbeault 1-derivations revisited]\label{ex:holrev} {\em 
Following Example~\ref{ex:hol}, a Dolbeault 1-deri\-va\-tion on $E\to M$ is a 1-derivation $(D, l, r)$ on  $E$ satisfying both the Nijenhuis equations  \eqref{eq:van_Nijenhuis} and the almost complex equations  \eqref{eq:square_minus_id}; note that these conditions exactly say that $(M,r)$ is a complex manifold and 
\begin{equation}\label{eq:flat}
\overline{\partial}_{X+\ii r(X)}:= -l(D_X)
\end{equation}
is a flat $T^{0,1}$-connection on the complex vector bundle $(E, l)$.  We will think of a holomorphic vector bundle $\mathcal{E}$ as a  pair
$$
\mathcal{E}= (E, \mathcal{D}^{Dolb})
$$
given by a real vector bundle $E\to M$ equipped with a Dolbeault 1-derivation $\mathcal{D}^{Dolb}$. 
}
\hfill $\diamond$
\end{ex}

%%%%%%%%%%%%%%%%%%%%%%%%%%%%%%%%%%%%%%%%%%%%%%%%%%%%%%%%%%%%%%%%%%%%
\subsection{Duality}
Just as usual connections, 1-derivations on a vector bundle $E\to M$ possess a duality operation that establishes a bijection between 1-derivations on $E$ and $E^*$. Given a 1-derivation $\mathcal{D} = (D,l,r)$ on $E$, its dual 1-derivation is $\mathcal{D}^*= (D^*, l^*, r)$, where $D^*: \Gamma(E^*) \to \Gamma(T^*M\otimes E^*)$ is characterized by
\begin{equation}\label{defn:dual_derivation}
 \<D^*_X(\mu), \sigma\> = \Lie_X\<\mu,l(\sigma)\> - \Lie_{r(X)}\<\mu, \sigma\> - \<\mu, D_X(\sigma)\>.
\end{equation}
The corresponding linear (1,1)-tensor on $E^*$ is denoted by 
$$
K^\top: T(E^*) \to T(E^*)
$$ 
and admits the following characterization. If $\llbracket \cdot, \cdot \rrbracket: TE \times_{TM} T(E^*) \to TM \times \R$ is the non-degenerate, symmetric, bilinear pairing obtained from differentiation of the natural pairing $\<\cdot, \cdot\>: E \times_M E^* \to M \times \R$, then
$$
\llbracket K(U), K^\top(V) \rrbracket = \llbracket Tl(U), V\rrbracket, \qquad \forall \, (U,V) \in TE\times_{TM} T(E^*).
$$

Applying duality in  Examples \ref{ex:connection}, \ref{ex:lifts} and \ref{ex:hol}, one obtains the following (see \cite[$\S$~2.1.2]{Dru} for details):
\begin{itemize}
    \item For the 1-derivation associated to a connection $\nabla$, the dual 1-derivation corresponds to the dual connection $\nabla^*$ on $E^*$.
    \item For $r: TM \to TM$, the 1-derivations $\mathcal{D}^r$ and $\mathcal{D}^{r,*}$
    (corresponding to $r^{tg}$ and $r^{ctg}$) are dual to each other.
    \item A 1-derivation $\mathcal{D}$ is Nijenhuis if and only if so is $\mathcal{D}^*$, see \cite{Dru}; in this case,  $\mathcal{D}$ satisfies the almost complex conditions if and only if so does $\mathcal{D}^*$. As a consequence, $\mathcal{D}$ is a Dolbeault 1-derivation if and only if so is $\mathcal{D}^*$. For a holomorphic vector bundle $\mathcal{E}$, the dual to its Dolbeault 1-derivation is the Dolbeault 1-derivation corresponding to the holomorphic structure of the dual $\mathcal{E}^*$, see Example~\ref{ex:holrevext} below.
\end{itemize}

\medskip

It will be  useful to extend, for each $X\in \mathfrak{X}(M)$, the operator $D^*_X$ to  $\Gamma(\wedge^m E^*)$ as follows:
\begin{align}
\label{eq:extension_f}\ D_X^*(f) = & -\Lie_{r(X)}f, \\
\label{eq:extension} D^*_X(\mu)(\sigma_1,\ldots,\sigma_m) =
& \Lie_X \mu(l( \sigma_1),\ldots,\sigma_m) -  \Lie_{r(X)} \mu(\sigma_1,\ldots,\sigma_m)    \\
\nonumber & - \sum_{k=1}^{m} (-1)^{k-1} \mu(D_X(\sigma_k), \sigma_1,\ldots, \widehat{\sigma_k},\ldots , \sigma_m), 
\end{align}
where $f \in C^\infty(M)$, and $m \geq 1$. 
Note that $D_X^*(\mu)$ does not define an element of $\Gamma(\wedge^m E^*)$ in general, since it is not $C^{\infty}(M)$-multilinear.
In general, $D_X(\mu)(\sigma_1, \cdot)$ is skew-symmetric as a map from $(m-1)$ copies of $\Gamma(E)$ to $C^{\infty}(M)$, and it satisfies
\begin{align*}
D_X^*(\mu)(f\sigma_1, \cdot) & = f D_X^*(\mu)(\sigma_1,\cdot), \\ 
D_X^*(\mu)(\sigma_1, f\sigma_2, \cdot) & = f D_X^*(\mu)(\sigma_1, \sigma_2, \cdot) + (\Lie_Xf) (\mu(l(\sigma_1), \sigma_2, \cdot) - \mu(\sigma_1, l(\sigma_2), \cdot)).
\end{align*}

Let us consider
$$
\Gamma_l(\wedge^m E^*) = \{\mu \in \Gamma(\wedge^m E^*) \,\, | \,\, \mu(l(\sigma_1), \sigma_2, \cdot) = \mu(\sigma_1, l(\sigma_2), \cdot), \,\, \forall \, \sigma_1, \, \sigma_2 \in \Gamma(E)\}.
$$ 
The next result follows immediately.

\begin{prop}
For $\mu \in \Gamma(\wedge^m E^*)$, $D_X^*(\mu) \in \Gamma(\wedge^m E^*)$ if and only if $\mu \in \Gamma_l(\wedge^m E^*)$. Moreover, if $[D_X,l] = D_X \circ l - l\circ D_X = 0$, then 
$D^*_X(\Gamma_l(\wedge^\bullet E^*) \subset  \Gamma_l(\wedge^\bullet E^*)$.
\end{prop}

Recall from \eqref{eq:van_Nijenhuis} that condition $[D_X,l]=0$ holds  if $\mathcal{D}$ is Nijenhuis. 

In the following, for $\mu \in \Gamma(\wedge^m E^*)$, we shall denote by $\mu_l$ the element of $\Gamma(E^* \otimes \wedge^{m-1} E^*)$ given by $\mu_l=0$ if $m=0$, and 
$$
\mu_l(\sigma; \sigma_1, \dots, \sigma_{m-1}) = \mu(l(\sigma), \sigma_1, \dots, \sigma_{m-1}),
$$
for $m \geq 1$, in such a way that $\mu \in \Gamma_l(\wedge^m E^*)$ if and only if $\mu_l \in \Gamma(\wedge^m E^*)$.

%\textcolor{blue}{INCLUDE? extension of $D_X$ to exterior algebra compatible with $l$....}

\begin{ex}\label{ex:holrevext}
{\em  For a holomorphic vector bundle $\mathcal{E}=(E, \mathcal{D}^{Dolb})$, with $\mathcal{D}^{Dolb}=(D, l, r)$, the dual 1-derivation $\mathcal{D}^{Dolb, *}=(D^*, l^*, r)$ is a Dolbeault 1-derivation on $E^*$ corresponding to the holomorphic structure on the complex vector bundle $(E^*,l^*)$ given by the $T^{0,1}$-connection $\overline{\partial}^*$ dual to $\overline{\partial}$, see \eqref{eq:flat}. Note also that $(E^*,l^*)$ is identified with the bundle of complex-linear functionals $(E,l)\to \C$ via
\begin{equation}\label{eq:Phi1}
\mu \mapsto \mu -\ii \mu_l.
\end{equation}
Let us denote the  complex exterior algebra bundle of $(E^*, l^*)$ by $\wedge_\C E^*$, so that $\Gamma(\wedge^m_\C E^*)$ is the bundle of complex-multilinear alternating $m$-forms on $(E,l)$, that carries a natural extension of $\overline{\partial}^*$.
The map \eqref{eq:Phi1} extends to an isomorphism 
$$
\Phi: \Gamma_l(\wedge^m E^*) \to \Gamma(\wedge^m_\C E^*), \qquad  \Phi(\mu) = \mu - \ii \mu_l,
$$
of $C^\infty(M,\C)$-modules (multiplication by $\ii$ on $\Gamma_l(\wedge^m E^*)$ corresponds to the operation $\mu\mapsto \mu_l$). This isomorphism satisfies
$$
\Phi(D_X^*(\mu))  = \ii \overline{\partial}^*_{X+\ii r(X)} \Phi(\mu) 
$$
showing that the extension of $D^*$ (as in \eqref{eq:extension_f} and \eqref{eq:extension}) matches that of $\overline{\partial}^*$. In particular, 
$\mu \in \Gamma_l(\wedge^m E^*)$ satisfies $D^*(\mu)=0$ if and only if $\mu - \ii \mu_l \in \Gamma(\wedge^m_\C E^*)$ is a holomorphic section. Hence, for each open subset $U\subseteq M$, elements in $\Gamma_l(\wedge^* E^*|_U)$ in the kernel of $D^*$ characterize holomorphic sections of 
$\Gamma(\wedge^m_\C E^*|_U)$ in terms of their real parts. 
%\textcolor{blue}{mention that $D^*$ is local, can consider local sections/sheaves}
%. ***** The complex multiplication on the left is given by $\mu\mapsto \mu_l$. One can check that
%$$
%\Phi(D_X(\mu)) = \ii \overline{\partial}_{X+\ii r(X)}(\Phi(\mu)).
%$$
% \textcolor{blue}{So, the space of solutions of $D(\mu)=0$, $\mu \in \Gamma_l(\wedge^* E^*)$, is the real counterpart of holomorphic sections of $\wedge^{(m,0)} E^*$,  }
}

\hfill $\diamond$
\end{ex}

%%%%%%%%%%%%%%%%%%%%%%%%%%%%%%%%%%%%%%%%%%%%%%%%%%%%%%%%%%%%%%%%%

%%%%%%%%%%%%%%%%%%%%%%%%%%%%%%%%%%%%%%%%%%%%%%%%%%%%%%%%%%%%%%%
\subsection{Compatibility with (pre-)Lie algebroids}\label{subsec:IM}
The discussion in this subsection will be used later in the context of Dirac structures and lagrangian splittings of Courant algebroids.

Let us a consider a vector bundle $A\to M$ equipped with an anchor map $\rho: A\to TM$ (i.e., a vector bundle map over the identity map on $M$) and an $\R$-bilinear, skew-symmetric bracket $[\cdot,\cdot]$ on $\Gamma(A)$ such that
$$
[a,fb] = f[a,b] + (\Lie_{\rho(a)}f) b,
$$
for $a, b \in \Gamma(A)$ and $f\in C^\infty(M)$. 
The triple $(A, \rho, [\cdot,\cdot])$ is called a {\em pre-Lie algebroid} \cite{graburb}. A Lie algebroid is a pre-Lie algebroid such that $[\cdot,\cdot]$ satisfies the Jacobi identity. 

Just as commonly done for Lie algebroids (see e.g. \cite{Mac-book}), on a pre-Lie algebroid the anchor $\rho$ and bracket $[\cdot,\cdot]$ can be encoded in a degree-1 derivation $d_A$ of $\Gamma(\wedge A^*)$, or, alternatively, in a linear bivector field $\pi_A$ on the total space of $A^*\to M$.
In terms of $d_A$ and $\pi_A$, Lie algebroids are characterized by the further conditions that $d_A^2=0$ or that $\pi_A$ is a Poisson structure.

\begin{defn}\label{def:LAcomp}
A 1-derivation $\mathcal{D}=(D, l, r)$ on a vector bundle $A\to M$ is compatible with a pre-Lie algebroid structure ($\rho, [\cdot, \cdot])$ if the following equations hold:
\begin{align}
\tag{IM1}\label{eq:im1_lie} \rho \circ l & = r \circ \rho\\
\tag{IM2}\label{eq:im2_lie} \rho(D_X(a)) & = D^r_X(\rho(a))\\
\tag{IM3} \label{eq:im3_lie} l([a,b]) & = [a,l(b)] - D_{\rho(b)}(a)\\
\tag{IM4} \label{eq:im4_lie} D_X([a,b]) & = [a,D_X(b)]+[D_X(a),b] +D_{[\rho(b),X]}(a) - D_{[\rho(a),X]}(b),
\end{align}
for all $X \in \frakx(M)$ and $a, b \in \Gamma(A)$.
\end{defn}

These compatibility equations have the following geometric interpretations. 
Let $K: TA \to TA$ be the linear (1,1)-tensor field corresponding to $\calD$.
\begin{itemize}
    \item The conditions in Def.~\ref{def:LAcomp} hold if and only if the bivector field $\pi_A$ on $A^*$ is compatible with the (1,1)-tensor field $K^\top: T A^* \to T A^*$ in the sense of Magri-Morosi \cite{MM}, see \cite[$\S$~4.3]{Dru} and Example \ref{ex:MM} below. In particular, when $A$ is a Lie algebroid and $\mathcal{D}$ is a Nijenhuis 1-derivation, they hold if and only if the pair $(\pi_A, K^\top)$ is a Poisson-Nijenhuis structure. 
    \item When $A$ is a Lie algebroid, the compatibility in Def.~\ref{def:LAcomp} says that $K:TA \to TA$ is a Lie algebroid morphism with respect to the tangent prolongation Lie algebroid $TA\to TM$ \cite[$\S$ 6]{HT}. Hence 1-derivations compatible with a Lie algebroid 
    are the infinitesimal counterparts of multiplicative (1,1)-tensor fields on Lie groupoids; for this reason (IM1)--(IM4) above are called {\em IM equations} (where IM stands for ``infinitesimally multiplicative''), 
    and 1-derivations satisfying them are also referred to as {\em IM (1,1)-tensors} on Lie algebroids. 
\end{itemize}

The compatibility of a 1-derivation $\calD$ with a pre-Lie algebroid structure on $A$ can be encoded using the dual 1-derivation $\calD^*$ and the operator $d_A$ on $\Gamma(\wedge A^*)$ as follows.

\begin{prop}\label{prop:dual_im}
    A 1-derivation $\calD=(D,l,r)$ is compatible with a pre-Lie algebroid structure $(\rho, [\cdot,\cdot])$ on a vector bundle $A \to M$ if and only if, for all $\mu \in \Gamma_l(\wedge^m A^*)$, the following holds:
    \begin{align}
 \label{eq:dual_im_a}       D_{\rho(a)}^*(\mu) & = i_a d_A(\mu_l) - i_{l(a)} d_A\mu, \\
 \label{eq:dual_im_b}       i_a d_AD_X^*(\mu) & = D_X^*(d_A\mu)(a; \,\cdot) + \mathcal{R}_X(\mu)(a; \, \cdot),
\end{align}
where $\mathcal{R}_X(\mu)(a; \cdot)$ is the $\R$-multilinear skew-symmetric map from $m$-copies of $\Gamma(A)$ to $C^{\infty}(M)$ defined by
\begin{align*}
\mathcal{R}_X(\mu)(a;\sigma_1,\dots, \sigma_m) = & \Lie_X(D_{\rho(a)}^*(\mu)(\sigma_1,\dots, \sigma_m))  + D_{[\rho(a),X]}^*(\mu)(\sigma_1,\dots, \sigma_m)\\
& - \sum_{i=1}^m (-1)^{i+1} D^*_{[\rho(\sigma_i),X]}(\mu)(a, \sigma_1, \dots, \widehat{\sigma_i}, \dots, \sigma_m).
\end{align*}
\end{prop}

\begin{proof}
For $\mu = f \in C^{\infty}(M)$, making use of \eqref{defn:Dr_dual} and \eqref{eq:extension_f}, one can check that \eqref{eq:dual_im_a} and \eqref{eq:dual_im_b} are equivalent to 
$$
\Lie_{l(\rho(a)) - \rho(r(a))}f = 0 \;\;\, \text{ and }\,\;\; D_X^*(d_Af)= \rho^*D_X^{r,*}(df),
$$
respectively. Using \eqref{defn:dual_derivation}, one can now check that  \eqref{eq:dual_im_a} and \eqref{eq:dual_im_b} in degree 0 are equivalent to \eqref{eq:im1_lie} and \eqref{eq:im2_lie}. Similarly, for $\mu \in \Gamma(A)$ and assuming  that \eqref{eq:im1_lie} holds, one can check that \eqref{eq:dual_im_a} is equivalent to \eqref{eq:im3_lie}. Finally, under the assumption that \eqref{eq:im1_lie}, \eqref{eq:im2_lie} and \eqref{eq:im3_lie} hold, one verifies that \eqref{eq:dual_im_b} is equivalent to \eqref{eq:im4_lie}. For higher degrees, \eqref{eq:dual_im_a} and \eqref{eq:dual_im_b} follow directly from the IM equations.
\end{proof}

Define 
$$
\Gamma_{\calD}(\wedge^m A^*) = \{\mu \in \Gamma_l(\wedge^m A^*)\,\, | \,\,D^*_X(\mu) =0, \,\, \forall \, X \in \frakx(M)\}.
$$
A direct consequence of Proposition \ref{prop:dual_im} is that, when $\calD$ is compatible with the a pre-Lie algebroid structure on $A$, 
$$
d_A(\Gamma_{\calD}(\wedge^m A^*)) \subseteq \Gamma_{\calD}(\wedge^m A^*),
$$
so $(\Gamma_\calD(\wedge^\bullet A^*),d_A)$ is a subcomplex of $(\Gamma(\wedge^\bullet A^*), d_A)$.

\begin{ex}[Holomorphic Lie algebroids] \label{ex:holLA} {\em Consider a Dolbeault 1-derivation $\mathcal{D}^{Dolb}$ on a vector bundle $A\to M$, so that $\mathcal{A}=(A, \mathcal{D}^{Dolb})$ is a holomorphic vector bundle. It is shown in \cite[$\S$ 6.4]{HT} that a Lie algebroid structure on $A\to M$ compatible with $\mathcal{D}^{Dolb}$ (in the sense of Def.~\ref{def:LAcomp}) is equivalent to a holomorphic Lie algebroid structure on $\mathcal{A}$ (in the sense of \cite[$\S$ 3.1]{LXS}). 
Under this correspondence, following Example~\ref{ex:holrevext}, we see that for each open subset $U\subseteq M$,  the complex 
$(\Gamma_\calD(\wedge^\bullet A^*|_U),d_A)$ is identified with the holomorphic Lie algebroid complex (see \cite[$\S$ 4.4]{LXS}) of $\mathcal{A}$ over $U$.
}
\hfill $\diamond$
\end{ex}
%%%%%%%%%%%%%%%%%%%%%%%%%%%%%%%%%%%%%%%%%%%%%%%%%%%%%%%%%%%%%%%%%
\section{1-Derivations on Courant algebroids}\label{sec:1derCA}
%In this section we use 1-derivations to introduce a notion of compatibility between Nijenhuis operators and Courant algebroids.

In this section we introduce a notion of compatibility between 1-derivations and Courant algebroids that is the main object of study in this paper.

\subsection{Courant 1-derivations and Courant-Nijenhuis algebroids}

We start by recalling Courant algebroids \cite{LiuXuWe,roytenberg}.

\begin{defn}\label{dfn:courant_alg} \em
A \textit{Courant algebroid} over a manifold $M$ is a vector bundle $E\rightarrow M$ together with a bundle map $\an: E \rightarrow TM$  (called the \textit{anchor}), a pseudo-euclidean metric $\< \cdot, \cdot\>$ (i.e, a fibrewise nondegenerate symmetric bilinear form), and an  $\R$-bilinear bracket $\Cour{\cdot, \cdot}: \Gamma(E)\times \Gamma(E)\rightarrow \Gamma(E)$ such that, for all $\sigma_1, \sigma_2, \sigma_3 \in \Gamma(E)$ and $f \in C^{\infty}(M),$ the following hold:
\begin{align}
\tag{C1}\label{C1} & \Cour{\sigma_1,\Cour{\sigma_2,\sigma_3}}  = \Cour{\Cour{\sigma_1,\sigma_2},\sigma_3} +\Cour{\sigma_2,\Cour{\sigma_1,\sigma_3}}  \\
\tag{C2}\label{C2} & \an(\Cour{\sigma_1,\sigma_2})  =[\an(\sigma_1),\an(\sigma_2)] \\
\tag{C3}\label{C3} & \Cour{\sigma_1,f\sigma_2} = f\Cour{\sigma_1,\sigma_2} + (\Lie_{\an(\sigma_1)}f)\,\sigma_2 \\
\tag{C4}\label{C4}& \Cour{\sigma_1,\sigma_2} +  \Cour{\sigma_2,\sigma_1}  = \an^*(d\< \sigma_1, \sigma_2\>)  \\
\tag{C5}\label{C5} & \; \Lie_{\an(\sigma_1)}\< \sigma_2, \sigma_3\>   = \< \Cour{\sigma_1,\sigma_2}, \sigma_3\> + \< \sigma_2, \Cour{\sigma_1, \sigma_3} \> 
\end{align}
where $\an^*:T^*M \rightarrow E^* \simeq E$ is the map dual to the anchor $\an$, and the isomorphism $E^* \simeq E$ is given by $\< \cdot, \cdot\>.$ We refer to $\Cour{\cdot, \cdot}$ as the {\it Courant bracket}.
\end{defn}

\begin{comment}
Note that if $(E, \an, \< \cdot, \cdot\>, \Cour{\cdot, \cdot})$ is a Courant algebroid, then $\Bar{E}=(E, \an, -\< \cdot, \cdot\>, \Cour{\cdot, \cdot})$ is also a Courant algebroid. Indeed, writing Equation $3)$ as $$\<\Cour{a_1,a_2} + \Cour{a_2,a_1}, a_3 \>= \an(a_3)\< a_1, a_2\>,$$ we see that $\Bar{E}$ satisfies all of conditions on the definition above. 
\end{comment} 
% Axioms $1)-3)$ imply some other important identities (we refer to \cite{uchino} for details):

\begin{comment} 
Equations $1)$ and $3)$ means that the Courant bracket satisfies Jacobi identity, but is not skew-symmetric. There is an equivalent definition of Courant algebroid in terms of a skew-symmetric bracket, but doesn't satisfy Jacobi identity (see \cite[Prop. 2.6.5]{DR}). The relation between the two bracket is given by
\begin{equation} \label{relationC}
   \Cour{a_1,a_2}_{skew}=\frac{1}{2}( \Cour{a_1, a_2}-\Cour{a_2,a_1}), 
\end{equation}
for all $a_i \in \Gamma(E).$
\end{comment}

The following is an important class of examples  \cite{severaletters,severaweinstein}.
 
\begin{ex}[Exact Courant algebroids]  \em \label{ex:Htwisted}
Any closed 3-form $H\in \Omega^3(M)$ defines a Courant algebroid structure on $E= \T M:= TM\oplus T^*M$, with anchor map $\an= \pr_{TM}$, symmetric pairing 
$$
\<(X,\alpha), (Y, \beta)\> = \beta(X) + \alpha(Y), 
$$
and the {\em $H$-twisted Courant bracket} 
$$
\Cour{(X,\alpha), (Y,\beta)} = ([X,Y], \Lie_X \beta - i_Y d\alpha + i_Yi_X H). 
$$
When $H=0$, one refers to this structure  as the {\em standard Courant algebroid} on $\T M$.
 \hfill $\diamond$

%The tangent bundle $TM$ with the Lie bracket of vector fields and anchor map $\rho = \mathrm{id}$ is a Lie algebroid that fits into a Lie bialgebroid $(TM, T^*M)$, where $T^*M$ has zero bracket and zero anchor. The double Courant algebroid $\T M := TM\oplus T^*M$ is the \textit{standard Courant algebroid} structure on $\T M$, given by $\an= \pr_{TM}$, 
%$$
%\<(X,\alpha), (Y, \beta)\> = \beta(X) + \alpha(Y), \;\; %\mathrm{and}\;\; \Cour{(X,\alpha), (Y,\beta)} = ([X,Y], \Lie_X \beta - i_Y d\alpha). 
%$$

\end{ex}

%\textcolor{blue}{introduce dirac and state correspondence of lie bialgebroids and manin triples here?}

%\textcolor{blue}{mention/remark Lie quasi-bialgebroids? exact courant algebroids? manin quasi-triples?}

%%%%%%%%%%%%%%%%%%%%%%%%%%%%%%%%%%%%%%%%%%%%%%%%%%%%%%%%%%%%%%
%\subsection{Courant 1-derivations}

Let $(E, \an, \<\cdot,\cdot\>, \Cour{\cdot, \cdot})$ be a Courant algebroid.

\begin{defn}\label{def:Cder} \em  A {\em Courant 1-derivation} is a 1-derivation $\mathcal{D} = (D,l,r)$ on the vector bundle $E\to M$ such that $\mathcal{D}^*=\mathcal{D}$ (under the identification $E \cong E^*$ given by the pairing) and the following compatibility equations are satisfied:
\begin{align}
 \tag{CN1}\label{eq:im1} &    \an \circ l = r \circ \an,\\
 \tag{CN2}\label{eq:im2} &    \an (D_X(\sigma))  = D_X^r(\an(\sigma)),\\
 \tag{CN3}\label{eq:im3} &     l(\Cour{ \sigma_1, \sigma_2 })= \Cour{ \sigma_1,l(\sigma_2) } - D_{\an(\sigma_2)}(\sigma_1) - \an^*( C(\sigma_1,\sigma_2)),\\ 
 \tag{CN4}\label{eq:im4} &    D_X(\Cour{ \sigma_1,\sigma_2 })  = \Cour{ \sigma_1,D_X(\sigma_2) } - \Cour{ \sigma_2,D_X(\sigma_1)} + D_{[\an(\sigma_2),X]}(\sigma_1)\\
 \nonumber   & \qquad  \qquad \qquad \quad - D_{[\an(\sigma_1),X]}(\sigma_2) - \an^*(i_X\,dC(\sigma_1,\sigma_2)),
\end{align}
for all $\sigma_1, \sigma_2 \in \Gamma(E)$ and $X \in \mathfrak{X}(M),$ where $C(\sigma_1,\sigma_2):= \<D_{(\cdot)}(\sigma_1) ,\sigma_2 \>  \in \Omega^1(M)$.
%is defined by $C= \<D_{(\cdot)}(\sigma_1) ,\sigma_2 \> \in \Omega^1(M).$ 
\end{defn}

Equations \eqref{eq:im1}-\eqref{eq:im4} are called \textit{Courant compatibility equations} for $\mathcal{D}$.
Note that they impose a linear condition on 1-derivations, so  linear combinations of Courant 1-derivations are still Courant 1-derivations.

%\textcolor{blue}{dirac structures compatible with Courant 1-der?}

\begin{defn} \em \label{defCN}
A {\em Courant-Nijenhuis} 1-derivation is a Courant 1-derivation  that is also a Nijenhuis 1-derivation, i.e., satisfies the Nijenhuis equations \eqref{eq:van_Nijenhuis}. A Courant algebroid equipped with a Courant-Nijenhuis 1-derivation is a  {\it Courant-Nijenhuis algebroid}. 
\end{defn}

\begin{ex} {\em  
Consider a 1-derivation $\mathcal{D}$ defined by a connection $\nabla$ on $E\to M$, as in Example~\ref{ex:connection}.
If it is a Courant 1-derivation, then the anchor $\an$ must be trivial (by (CN1)), so that, as a Courant algebroid, $E$ is a bundle of quadratic Lie algebras. In this case, $\nabla$ is a Courant 1-derivation if and only if it is symmetric (i.e., $\nabla=\nabla^*$ with respect to the pseudo-euclidean metric) as well as
 compatible with the metric and fibrewise Lie bracket (by (CN4)):
$$
\Lie_X\< \sigma_1,\sigma_2\>=\<\nabla_X \sigma_1,\sigma_2\> + \<\sigma_1,\nabla_X \sigma_2\>, \qquad 
 \nabla_X(\Cour{ \sigma_1,\sigma_2 })  = \Cour{ \nabla_X\sigma_1,\sigma_2 } + \Cour{ \sigma_1,\nabla_X \sigma_2},
$$
for $X\in \mathfrak{X}(M)$. 

The Nijenhuis condition on $\mathcal{D}$ amounts to the flatness of $\nabla$. In this case, assuming that $M$ is connected and viewing its universal cover $\widetilde{M}$ as a $\pi_1(M)$-principal bundle over $M$, $E$ is of the form $(\widetilde{M}\times \mathfrak{d})/\pi_1(M)$, where $\mathfrak d$ is a quadratic Lie algebra equipped with a  representation of $\pi_1(M)$ that preserves bracket and pairing.
}
\hfill $\diamond$

\end{ex}

The next example from \cite{BDN} is the original motivation for the 
 Courant compatibility equations (CN1)--(CN4).

\begin{ex}\label{ex:Dr} \em For a (1,1)-tensor field $r:TM \rightarrow TM $, consider the 1-derivations $\mathcal{D}^r$ and $\mathcal{D}^{r,*}$ from  \eqref{defn:Dr} and \eqref{defn:Dr_dual}. Then, setting
\begin{equation*}
\D^r := (D^r, D^{r,*}),
\end{equation*}
we obtain a 1-derivation 
\begin{equation}\label{eq:doubleDr}
\mathcalbb{D}^r= (\D^r, (r,r^*), r)
\end{equation}
on  $\T M$
that is a Courant 1-derivation with respect to the standard Courant algebroid structure (this is verified in the proof of \cite[Lem.~6.1]{BDN}). 
More generally, one can check that  $\mathcalbb{D}^r$ is a Courant 1-derivation with respect to an $H$-twisted Courant bracket as long as the closed 3-form $H\in \Omega^3(M)$ is compatible with $r$ in the sense of \cite[Def.~5.1]{BDN}, i.e., 
\begin{itemize}
\item the tensor field $H_r\in \Gamma(T^*M\otimes \wedge^2 T^*M)$, $H_r(X_1;, X_2, X_3):= H(r(X_1), X_2, X_3)$, is skew-symmetric, that is, $H_r\in \Omega^3(M)$;
\item $d H_r =0$.
\end{itemize}

Moreover,  $r$ is a Nijenhuis operator if and only if $\mathcalbb{D}^r$ is a Nijenhuis 1-derivation \cite[Lem.~6.1]{BDN}, in which case it defines a
 Courant-Nijenhuis structure on $\T M$. One can also directly verify that $r^2=-\mathrm{Id}_{TM}$ if and only if 
$\mathcalbb{D}^r$ satisfies the almost complex equations in \eqref{eq:square_minus_id}. When $r$ is a complex structure on $M$, $\mathcalbb{D}^r$ is the Dolbeault 1-derivation codifying the holomorphic structure on $\T M \to M$ (see Example~\ref{ex:hol}), with corresponding linear complex structure given by
\begin{equation}\label{eq:Kr}
K_r: = (r^{tg}, r^{ctg}): T (\T M) \to T (\T M). 
\end{equation}
\hfill $\diamond$
\end{ex}

We will describe in $\S$ \ref{sec:T+T*} modifications of the last example yielding more general Courant 1-derivations on $\T M$.

%\textcolor{blue}{B-field invariance of Courant 1-derivations comes çater... if $B$ satisfies an extra compatibility condition (B is Dirac Nij)}

As we will see in $\S$ \ref{sec:hol}, for Dolbeault 1-derivations (see Example \ref{ex:hol}), the compatibility with Courant structures yields {\em holomorphic Courant algebroids}.

\begin{rmk}[On the compatibility equations]{\em 
A natural issue concerning Courant 1-deriva\-tions is whether the Courant compatibility equations  for  $\mathcal{D}$ in Def.~\ref{def:Cder} admit a geometric interpretation in terms of the corresponding linear (1,1)-tensor field $K$ (c.f. the discussion after Def.~\ref{def:LAcomp}). Although a satisfactory answer does not seem evident (in contrast with Lie algebroids, one can check that, in general, $K$ does not define a Courant morphism \cite{BIS} of the tangent Courant algebroids \cite{BouZaa}),
some key properties of $K$ will be presented  in Prop.~\ref{prop:Dirac_K} below. From another perspective, it would be interesting to see if the super-geometric viewpoint on Courant algebroids \cite{roytenberg} can shed light on the Courant compatibility equations.
%$\S$ \ref{subsec:dirac} below.
}
\hfill $\diamond$
\end{rmk}

\begin{comment}
Let us begin by characterizing the Courant-Nijenhuis structures via their corresponding linear (1,1)-tensors. First recall that the tangent prolongation bundle $TE \to TM$ of a Courant algebroid $E \to M$ has itself a Courant algebroid structure (see \cite{LB} and references therein). We shall use the notation $\overline{E}$ to refer to the Courant algebroid obtained by multiplying the pairing $E$ by $-1$. 

\begin{prop}\label{prop:Cour_morphism}
Let $(D,l,r)$ be a 1-derivation on a Courant algebroid $E \to M$ and $K: TE \to TE$ the corresponding linear $(1,1)$-tensor. One has that $(D,l,r)$ is a Courant 1-derivation if and only if $K^\top=K$ and the following two conditions hold:
\begin{enumerate}
       \item $\an_{TE\times T(\overline{E})}(\mathrm{Graph}(K)) \subset T\mathrm{Graph}(r)$, where $\an_{TE\times T(\overline{E})}$ is the anchor of the Courant algebroid $TE \times T(\overline{E}) \to TM \times TM$.
    \item For $\sigma_1, \sigma_2 \in \Gamma(TM\times TM, TE\times T(\overline{E}))$ such that $\sigma_i|_{\mathrm{Graph}(r)} \in \Gamma(\mathrm{Graph}(K))$, $i=1,2$, 
    $$
    \Cour{\sigma_1,\sigma_2}|_{\mathrm{Graph}(r)} \in \Gamma(\mathrm{Graph}(K)).
    $$
\end{enumerate}
Moreover, $(D,l,r)$ satisfies the Nijenhuis equations if and only if $\N_K=0$.
\end{prop}
\end{comment}

%We shall prove Proposition \ref{prop:Cour_morphism}  in Appendix \ref{app:proof}.

%%%%%%%%%%%%%%%%%%%%%%%%%%%%%%%%%%%%%%%%%%%
\subsection{Invariant Dirac structures}\label{subsec:dirac}
%\textcolor{blue}{move here: Dirac nijenhuis}
Much of the importance of Courant algebroids lies in their Dirac structures. We now consider Dirac structures invariant by Courant 1-derivations.

Let $\mathcal{D}=(D, l, r)$ be a Courant 1-derivation on a Courant algebroid $E\to M$. We will be concerned with lagrangian subbundles $(L\to M)\subseteq (E\to M)$ (i.e., $L=L^\perp$) that are $\mathcal{D}$-invariant in the sense of Def.~\ref{def:compat}:
\begin{itemize}
    \item $l(L) \subset L$, and 
    \item $D_X(\Gamma(L)) \subset \Gamma(L),\; $ for all $X\in \mathfrak{X}(M).$ 
\end{itemize}

For lagrangian subbundles, these conditions can be equivalently expressed as follows. Consider the symmetric 2-form $S\in \Gamma(S^2E^*)$,
$$
S(\sigma_1,\sigma_2) = \< l(\sigma_1), \sigma_2\>.
$$
(Recall that $l=l^*$ since $\mathcal{D}^*=\mathcal{D}$.) Let $C:\Gamma(E) \times \Gamma(E) \to \Omega^1(M)$ be the map that appeared in the definition of Courant 1-derivations,  
$$
C(\sigma_1,\sigma_2) = \<D_{(\cdot)}(\sigma_1), \sigma_2\>.
$$ 
For a lagrangian subbundle $L \subset E$, the restriction of $S$ to $L$ defines an element $S_L \in \Gamma(S^2 L^*)$ with the property that 
$l(L)\subseteq L$ if and only if $S_L=0$.
Assuming that this holds, the restriction of $C$ to sections of $L$ is tensorial, i.e., it  defines an element 
$$
C_L \in \Gamma(\wedge^2 L^* \otimes T^*M)
$$ 
called the \textit{concomitant of $L$ and $\mathcal{D}$}. The following result can be directly verified.

\begin{lem}\label{lem:lagr_comp}
For a Courant 1-derivation $\mathcal{D}$, 
a lagrangian subbundle $L\subseteq E$ is $\calD$-invariant if and only if 
$$
S_L=0, \; \mbox{ and } \; C_L=0.
$$
\end{lem}

Suppose that $\mathcal{D}$ is a Courant-Nijenhuis 1-derivation on $E$, so that $(E,\mathcal{D})$ is a Courant-Nijenhuis algebroid.

\begin{defn} \em \label{defDN}
 A Dirac structure $L\subset E$ that is $\mathcal{D}$-invariant (equivalently, such that $S_L=0$ and $C_L=0$)
is called a \textit{Dirac-Nijenhuis structure}. \end{defn}

\begin{ex}[Poisson-Nijenhuis structures] \label{ex:MM}\em 
Given a (1,1)-tensor field $r: TM \to TM$, consider the associated Courant 1-derivation on $\T M$ (with the standard Courant bracket) given by 
$
\mathcalbb{D}^r,
$ 
as in Example~\ref{ex:Dr}. Its invariant lagrangian subbundles are precisely those considered in \cite[$\S$ 3.3]{BDN}. In particular, a lagrangian subbundle given by the graph of a bivector field $\pi \in \frakx^2(M)$, 
$$
L_\pi = \{(\pi^\sharp(\xi), \xi) \in \T M\,\, | \,\, \xi \in T^*M\},
$$ 
where $\pi^\sharp: T^*M \to TM$ is the map obtained via contraction, is $\mathcalbb{D}^r$-invariant if and only if  $\pi$ and $r$ satisfy
\begin{equation}\label{eq:compatpir}
\; r\circ \pi^\sharp = \pi^\sharp \circ r^*, \; \mbox{ and } \;  \pi^\sharp \circ D^{r,*}_X(\alpha) - D^{r}_X\circ \pi^\sharp(\alpha)  = \pi^\sharp(\Lie_X r^*\alpha - \Lie_{r(X)}\alpha)-(\Lie_{\pi^\sharp(\alpha)}r) (X)=0
\end{equation}
for all $X\in \mathfrak{X}(M)$, $\alpha\in \Omega^1(M)$.
The expression in the second condition is known as the {\em Magri-Morosi concomitant} of $r$ and $\pi$ \cite{MM}, see \cite[$\S$ 3.1]{BDN}. 
From the viewpoint of Lemma \ref{lem:lagr_comp}, using the natural identification $L_\pi \cong T^*M$,  we have that 
$$
S_{L_\pi}(\alpha,\beta) = \pi(r^*\alpha,\beta) -  \pi(\alpha, r^*\beta), \quad \mbox{ and } \quad
C_{L_\pi}(\alpha,\beta) = \<\beta, \pi^\sharp \circ D^{r,*}_{(\cdot)}(\alpha) - D^{r}_{(\cdot)}\circ \pi^\sharp(\alpha)\>.
$$
(The alternative formulation of the vanishing of the Magri-Morosi concomitant in terms of $C_{L_\pi}$ goes back to \cite{KS-Magri90} and is now more frequent in the literature.)

Recall that $\pi$ is Poisson if and only $L_\pi$ is a Dirac structure, and $r$ is a Nijenhuis operator if and only if $\mathcalbb{D}^r$ is a Courant-Nijenhuis 1-derivation, so $L_\pi$ is a Dirac-Nijenhuis structure in $(\T M, \mathcalbb{D}^r)$ if and only if $(\pi,r)$ is a Poisson-Nijenhuis structure \cite[Ex.~3.11]{BDN}.
%\textcolor{blue}{spell out Magri morosi compatibility...}
 \hfill $\diamond$
\end{ex}

We refer to \cite{BDN} for more on Dirac-Nijenhuis structures in the specific Courant-Nijenhuis algebroid $\T M$ of the previous example.

It is well known that if $L$ is a Dirac structure in a Courant algebroid $E$, then the restrictions of the anchor and Courant bracket make $L$ into a Lie algebroid. In the presence of a Courant 1-derivation $\calD$ for which $L$ is $\calD$-invariant, it is a straightforward verification that  the 1-derivation on $L$ obtained by restriction of $\calD$ is compatible with its Lie algebroid structure, in the sense of Def.~\ref{def:LAcomp}. Following the discussion in $\S$ \ref{subsec:IM}, in terms of linear (1,1)-tensor fields we have

%\begin{lem} Suppose that a Dirac structure $L\subset E$ is compatible with a Courant 1-derivation $(D, l, r)$. Then the restricted 1-derivation on $L$ satisfies the IM-equations (IM1)--(IM4) with respect to the Lie algebroid structure of $L$.
%\end{lem}

%\textcolor{blue}{later: include examples for other courant 1-derivations on standard courant algebroid?}

%\textcolor{blue}{ 
%\begin{rmk}\em
%The conditions (1) and (2) on $K$ are part of the necessary conditions for  $K$ to be a Courant morphism from $TE \to TM$ to $TE \to TM$ \cite{}. The remaining condition (i.e. that $\mathrm{Graph}(K)$ must be isotropic) is equivalent to both $l=\mathrm{id}_{E}$ and $K^\top = K$. So, for an arbitrary Courant 1-derivation, its associated linear (1,1)-tensor is not a Courant morphism.
%\end{rmk}}

%Let us discuss how a Dirac-Nijenhuis structure can be characterized via linear (1,1)-tensors. First recall that, given a Dirac structure $L \subset E$, the restriction of the anchor and Courant bracket to $L$ endows it with a Lie algebroid structure.  

\begin{prop}\label{prop:Dirac_K}
Consider a Courant 1-derivation $\mathcal{D}$ on a Courant algebroid $E$, and let $K: TE\to TE$ be the corresponding linear (1,1)-tensor field. Then a lagrangian subbundle $L \subset E$ is $\mathcal{D}$-invariant if and only if $K(TL) \subset TL$. If $L$ is, in addition, a Dirac structure then $K|_{TL}: TL \to TL$ is a Lie algebroid morphism, where $TL \to TM$ is the tangent prolongation Lie algebroid. 
\end{prop}

\begin{proof}
As recalled in $\S$ \ref{subsec:1,1}, the fact that $L$ is $\mathcal{D}$-invariant is equivalent to $K(TL) \subset TL$, see \cite[Thm.~2.1]{BDN}. 
When $L$ is a Dirac structure, the restricted 1-derivation satisfies the IM equations from Def.~\ref{def:LAcomp}, which are equivalent to $K|_{TL}: TL \to TL$ being a Lie-algebroid morphism \cite[$\S$~5.2]{HT}.
\end{proof}

\section{Courant 1-derivations on $TM\oplus T^*M$}\label{sec:T+T*}
 We now present other examples of Courant 1-derivations on  $\T M =TM \oplus T^*M$ with respect to twisted Courant brackets, extending Example~\ref{ex:Dr}

\subsection{Courant 1-derivations from pseudo-Riemmanian metrics}

A direct calculation shows that a 1-derivation $\mathcal{D}=(D,l,r)$ on $\T M$ satisfying $\mathcal{D}^*=\mathcal{D}$ and equations \eqref{eq:im1} and \eqref{eq:im2} must have the form 
\begin{equation}\label{dfn:std_der}
l= (r,r^*+g^\flat), \,\,\,\, D = (D^r, D^{r,*} + \Sigma),
\end{equation}
where $g^\flat: TM \to T^*M$, $g^\flat(X)=i_Xg$, is the map obtained by contraction of a symmetric bilinear form $g$, and $\Sigma: \Gamma(TM) \to \Gamma(T^*M\otimes T^*M)$ is an $\R$-linear map satisfying
%\begin{equation*}
%\left\{
%\begin{matrix}
\begin{align}
\label{eq:Sigma_prop1}\Sigma_X(fY) & = f\,\Sigma_X(Y) + (\Lie_X f) \,g^\flat(Y),\\ %\vspace{4pt}\\
\label{eq:Sigma_prop2}\Lie_X g(Y,Z) & = \<\Sigma_X(Y), Z\> + \<Y, \Sigma_X(Z)\>,
\end{align}
%\end{matrix}
%\right.
%\end{equation*}
where $f\in C^\infty(M)$, $X,Y, Z \in \mathfrak{X}(M)$, and $\Sigma_X: \Gamma(TM)\to \Gamma(T^*M)$ is given by $\Sigma_X(Y)= i_X(\Sigma(Y))$.
In particular, any Courant 1-derivation $\calD=(D, l,r)$ on $\T M$ is determined by the data $r$, $g$ and $\Sigma$ as above, via \eqref{dfn:std_der}.
%, we shall say $\calD$ \textit{is determined by $(r,g,\Sigma)$} when $l$ and $D$ are given by \eqref{dfn:std_der}. 
When $g$ is non-degenerate, i.e., when $g$ is a {\em pseudo-Riemannian metric}, note that $\nabla_X = (g^\flat)^{-1}\circ \Sigma_X$ defines a metric connection on $M$ (i.e., $\nabla g=0$). 

In the following, { we shall assume that $g$ is a pseudo-Riemannian metric and $\nabla$ is its Levi-Civita connection}. In this case  $g$ and $\nabla$ give rise to a 1-derivation $\mathcalbb{D}^g = (\D^g, (0,g^\flat), 0)$ of $\T M$, where 
$$
\D^g_X((Y,\beta)) = (0, g^\flat(\nabla_X Y)),
$$ 
by setting $r=0$ in \eqref{dfn:std_der}. The corresponding linear (1,1)-tensor field on $q: \T M \to M$, 
\begin{equation}\label{eq:K_g}
K_g: T(\T M)\to T(\T M),
\end{equation}
is described as follows. The connection $\nabla$ defines a horizontal distribution $\mathrm{Hor}\subset T (\T M)$, complementary to the vertical distribution $\mathrm{Ver}=\mathrm{ker}(Tq)= q^* \T M$. With respect to this splitting of $T (\T M)$, $K_g$ vanishes of on $\mathrm{Hor}$ and acts as $(0, g^\flat)$ on $\mathrm{Ver}$. 

\begin{lem}
The 1-derivation $\mathcalbb{D}^g$  is a Courant 1-derivation of $\T M$ for any $H$-twisted Courant bracket.
\end{lem}

%\begin{equation}\label{dfn:metric_derivation}
%    \D^{r,g}_X (Y,\beta) = (D_X^r(Y), D_X^{r,*}( \beta) + g^\flat(\nabla_X Y))
%\end{equation}

\begin{proof}
Since $\mathcalbb{D}^g$ is symmetric, it remains to show that the Courant equations (CN1)--(CN4) are satisfied.
%Since $(\D^r, (r,r^*), r)$ is a Courant 1-derivation, it suffices to prove that $(D^g, (0,g^\flat), 0)$ is a Courant 1-derivation, where 
%$$
%D^g_X(Y) = (0, g^\flat(\nabla_X Y)).
%$$ 
%\textcolor{blue}{say something about linearity of Courant equations, sum of Courant derivations?} 
The only non-trivial equations to check are \eqref{eq:im3} and \eqref{eq:im4}. Proving \eqref{eq:im3} amounts to verifying that
\begin{equation}\label{eq:derivative_g}
g^\flat([X,Y]) = \Lie_X g^\flat(Y) - g^\flat(\nabla_Y X) - g(\nabla_{(\cdot)} X, Y), \qquad \forall \, X, Y \in \frakx(M).
\end{equation}
This holds because $\nabla$ is Levi-Civita: since $\nabla$ is metric and has zero torsion, we have
\begin{align*}
i_Z\Lie_X g^\flat(Y) & = \Lie_X g(Y,Z) - g(Y, [X,Z]) 
%& = g(\nabla_X Y, Z)+ g(Y, \nabla_X Z) - g(Y, \nabla_X Z - \nabla_Z X)\\
 = g(\nabla_Y X + [X,Y], Z) + g(Y,\nabla_Z X)\\
& = i_Z \left(g^\flat(\nabla_YX) + g^\flat([X,Y]) + g(Y,\nabla_{(\cdot)} X)\right).
\end{align*}
Condition \eqref{eq:im4}, in turn, follows from the first Bianchi identity. Indeed, one must show that
\begin{align*}
g(\nabla_Z[X,Y],W)  = & \<\Lie_X g^\flat(\nabla_Z Y), W\> - \<\Lie_Y g^\flat(\nabla_Z X), W\> + g(\nabla_{[Y,Z]}X, W) - g(\nabla_{[X,Z]}Y, W)\\
& - i_W i_Z d g(\nabla_{(\cdot)} X, Y)
\end{align*}
Let us denote the right-hand side of this last equation by $\Upsilon$.  Using that 
$$
i_W i_Z d g(\nabla_{(\cdot)} X, Y) = \Lie_Z g(\nabla_W X, Y) - \Lie_W g(\nabla_Z X, Y) - g(\nabla_{[Z,W]}X, Y)
$$ 
and, once again, the fact that $\nabla$ is metric and has zero torsion, one obtains that
\begin{align*}
\Upsilon & = g(R(X,Z)(Y), W) - g(R(Y,Z)(X), W) - g(R(Z,W)(X), Y) + g(\nabla_Z \nabla_X Y - \nabla_Z \nabla_Y X, W)\\
             & = \underbrace{g(R(X,Z)(Y), W) + g(R(Z,Y)(X), W) + g(R(Y,X)(Z), W)}_{=0 \; \text{(by the first Bianchi identity)}} + g(\nabla_Z [X,Y], W),
\end{align*}
where $R$ is the curvature tensor, and we used its symmetries in the last equality. This concludes the proof. 
\end{proof}

Now let
$$
\mathcalbb{D}^{r,g} = (\D^{r,g}, (r,r^*+g^\flat), r)
$$ 
be the 1-derivation defined by \eqref{dfn:std_der} with $\Sigma= g^\flat (\nabla)$, i.e., the sum of $\mathcalbb{D}^r$ (see Example~\ref{ex:Dr}) with the 1-derivation $\mathcalbb{D}^g$.

\begin{prop}\label{prop:Dg}
The 1-derivation $\mathcalbb{D}^{r,g}$ is a Courant 1-derivation on $\T M$ for any $H$-twisted Courant bracket such that $H$ is compatible with $r$.
\end{prop}

(The compatibility of $H$ and $r$ is recalled in Example~\ref{ex:Dr}).

\begin{proof}
The result follows from the previous lemma, the property that the sum of Courant 1-derivations is a Courant 1-derivation, and the fact that $\mathcalbb{D}^r$ is a Courant 1-derivation with respect to any $H$-twisted Courant bracket such that $H$ is compatible with $r$. 
\end{proof}

\begin{rmk}\em
More generally, one can replace the Levi-Civita connection
by any other metric connection $\nabla$ in the definition of $\D^{r,g}$. In this case, equations \eqref{eq:im3} and \eqref{eq:im4} are equivalent to the torsion $T$ of $\nabla$ being skew-symmetric and closed (i.e., $\varphi(X,Y,Z) := g(T(X,Y),Z)$ defines a closed 3-form).
\hfill $\diamond$
\end{rmk}

 \begin{ex} \label{ex:dirac_g}\em 
 {
The lagrangian subbundles $L \subset \T M$ which are $\mathcalbb{D}^g$-invariant are characterized, following Lemma \ref{lem:lagr_comp},  by the vanishing of $S_L \in \Gamma(S^2 L^*)$ and $C_L \in \Gamma(\wedge^2 L^* \otimes T^*M)$  given by
$$
S_L(\sigma_1,\sigma_2) = g(\pr_{TM}(\sigma_1),\pr_{TM}(\sigma_2)) \,\,\text{ and } \,\, C_L(\sigma_1,\sigma_2,Z) = g(\nabla_Z \,\pr_{TM}(\sigma_1), \sigma_2).
$$
The vanishing of $S_L$ is equivalent to the presymplectic distribution of $L$ being isotropic with respect to $g$. In particular, the only 2-form whose graph is $\mathcalbb{D}^g$-invariant is the zero 2-form. Under the assumption that $S_L=0$, a sufficient condition for the vanishing of $C_L$ is the invariance of $\pr_{TM}(L)$ with respect to the Levi-Civita connection (and these conditions are equivalent when $\pr_{TM}(L)$ is maximally isotropic).}
\hfill $\diamond$
%*** When this distribution is maximally isotropic with respect to $g$, $C_L=0$ if and only if the presymplectic distribution is Levi-Civita invariant.}
\end{ex}

\subsection{Nijenhuis 1-derivations  and the K\"ahler condition}

Let $r: TM\to TM$ be a (1,1)-tensor field, and let $g$ be a pseudo-Riemannian metric with Levi-Civita connection $\nabla$. We now give conditions on the pair $(r,g)$ ensuring that the 1-derivation
$\mathcalbb{D}^{r,g}$ satisfies the Nijenhuis equations, thereby defining a Courant-Nijenhuis 1-derivation (by Prop.~\ref{prop:Dg}).
We will denote the corresponding linear (1,1)-tensor field on $\T M$ by 
$$
K_{r,g} = K_r + K_g,
$$
see \eqref{eq:Kr} and \eqref{eq:K_g}. 
%\textcolor{blue}{sum as 1,1 tensors}

As recalled in Example~\ref{ex:Dr}, the 1-derivation $\mathcal{D}^r$ satisfies the Nijenhuis equations  if and only if $r$ is a Nijenhuis operator ($\N_r=0$).  On the other hand, it is a simple verification that $\mathcalbb{D}^g$ is always a Nijenhuis 1-derivation. More generally, we have

\begin{prop}\label{prop:Nij}
For a pair $(r,g)$, suppose that
\begin{itemize}
\item $r$ is a Nijenhuis operator,
\item $g^\flat \circ r= - r^* \circ g^\flat$,
\item $\nabla r =0$.
\end{itemize}
Then the 1-derivation $\mathcalbb{D}^{r,g}$ satisfies the Nijenhuis equations (equivalently, $K_{r,g}$ has vanishing Nijenhuis torsion).
\end{prop}

\begin{proof} 
Since  $\N_r=0$ by assumption, one only has to check the remaining two Nijenhuis equations in \eqref{eq:van_Nijenhuis}. Using that $\calD^r$ is already a Courant-Nijenhuis 1-derivation on $\T M$ (see Example \ref{ex:Dr}), it suffices to show that
\begin{align}
\label{eq:g_Nij1} g^\flat(D^r_Z(X)) - D^{r,*}_Z(g^\flat(X)) =\, & g^\flat(\nabla_Z\,r(X))-r^*(g^\flat(\nabla_ZX),  \\
\label{eq:g_Nij2}D^{r,*}_Y (g^\flat(\nabla_X Z)) - D^{r, *}_X( g^\flat(\nabla_Y Z)) + r^*g^\flat(\nabla_{[X,Y]}Z)  =\, & g^\flat\left(\nabla_X D^r_Y(Z) - \nabla_Y D^r_X(Z)  \vphantom{D^r_{[X,Y]}(Z)}\right.\\
\nonumber & \left. - D^r_{[X,Y]}(Z) - \nabla_{[X,Y]_r}Z \right).
\end{align} 
Using \eqref{defn:Dr} and \eqref{defn:Dr_dual}, the compatibility $g^\flat\circ r = -r^*\circ g^\flat$ and that $\nabla$ is Levi-Civita, we can prove that \eqref{eq:g_Nij1} is equivalent to
$$
g^\flat((\nabla_Xr)(Z)) + g((\nabla_{(\cdot)} r)(Z), X))=0,
$$
which holds since $\nabla r=0$.
%where we have used that $r(D^{r}_Z(X))-D^r_Z(r(X)) = N_r(X,Z)=0$ and the similar equation with $r^*$ and $D^{r,*}_Z$. 
%Using $r^*\circ g^\flat = -g_{\sharp}\circ r$ together with the vanishing of the torsion of $\nabla$, one obtains that
%\begin{align*}
%\Theta 
%& = g^\flat(-r(\nabla_Z X) + [X,r(Z)]-r([X,Z]) + \nabla_{r(Z)}X) + g(\nabla_\cdot Z, r(X))+  g(\nabla_\cdot r(Z), X)\\
%  & = g^\flat(-r(\nabla_X Z) + \nabla_X r(Z)) + g(\nabla_\cdot r(Z)- r(\nabla_\cdot Z), X)\\
%   & = g((\nabla_Xr)(Z)) + g((\nabla_\cdot r)(Z), X)=0,
%\end{align*}

Regarding \eqref{eq:g_Nij2}, we can use that $\nabla r =0$ and that $\nabla$ is Levi-Civita to rewrite some of its terms as follows:
\begin{align*}
\nabla_X D^r_Y(Z) & =  \nabla_X r(\nabla_Y Z) -  \nabla_X \nabla_{r(Y)} Z,\\
D^{r}_{[X,Y]}(Z) & = r(\nabla_{[X,Y]}Z)-  \nabla_{r([X,Y])} Z,  \\
D^{r,*}_X (g^\flat(\nabla_Y Z)) & = -g^\flat(\nabla_X \,r(\nabla_Y Z)) - g^\flat(\nabla_{r(X)}\nabla_Y Z).
\end{align*}
Similar formulas hold for $\nabla_Y D^r_X(Z)$ and $D^{r,*}_Y( g^\flat(\nabla_X Z))$. After some cancellations and re-groupings, one obtains that  \eqref{eq:g_Nij2} is equivalent to
\begin{align*}
g(R(Z, \cdot)(r(X))-r(R(Z,\cdot)(X)), Y) =0,
\end{align*}
where  we have used the symmetries of the curvature tensor $R$. Therefore \eqref{eq:g_Nij2} holds, since $[R,r]=0$. This concludes the proof.
\end{proof}

A pair $(r, g)$, where $r$ is a (1,1)-tensor field and $g$ is a pseudo-Riemannian metric on $M$, defines a {\em  (pseudo-) K\"ahler structure} if $r$ and $g$ satisfy the three conditions in Prop.~\ref{prop:Nij} with the additional requirement that $r^2=-\mathrm{Id}_{TM}$ (so that $r$ is a complex structure). Pseudo-K\"ahler structures  admit the following characterization in terms of 1-derivations.

\begin{prop}\label{prop:almostcomplex}
A pair $(r,g)$ defines a pseudo-K\"ahler structure if and only if the 1-derivation $\mathcalbb{D}^{r,g}$ satisfies the almost-complex equations in \eqref{eq:square_minus_id} (equivalently, $K_{r,g}^2 = -\mathrm{Id}$).
\end{prop}

\begin{proof} Recall that the almost-complex equations are 
$$
r^2=-\mathrm{id}_{TM},\;\;\;\;
(r,r^*+g^\flat)^2=-\mathrm{id}_{T(\T M)},\;\;\;\;\;
\D^{r,g}_{r(X)}(\sigma) + (r,r^*+g^\flat)(\D^{r,g}_X(\sigma)) = 0
$$
By the first equation, $r$ is an almost complex structure. The second equation is equivalent to
$r^*\circ g^\flat =- g^ \flat \circ r$ (i.e.,
$(g,r)$ is almost Hermitian). Since $r^2=-\mathrm{id}_{TM}$, the 1-derivation $(\D^r,(r,r^*), r)$ satisfies the almost complex equations (see Example~ \ref{ex:Dr}), and using this fact one verifies that
$$
\D^{r,g}_{r(X)}(Y,\beta) + (r,r^*+g^\flat)(\D^{r,g}_X(Y,\beta)) = (0,  \<\beta, \N_r(X,\cdot)\> + g^\flat((\nabla_Y r)(X))\,).
$$
By setting $Y=0$ or $\beta =0$ above, we see that the third almost-complex equation is equivalent to $\N_r=0$ and $\nabla r=0$. 
\end{proof}

It follows from Propositions~\ref{prop:Nij} and \ref{prop:almostcomplex} that, for a pair $(r,g)$, if $K_{r,g}$ is an almost complex structure then it is automatically integrable. 

%\begin{cor} A linear almost complex structure on $\T M$ is automatically integrable (so it defines a holomorphic structure on $\T M\to M$).
%\end{cor}

\begin{cor}\label{cor:Krg} A pseudo-K\"ahler metric $g$ on  a complex manifold $(M,r)$
defines a holomorphic vector bundle structure on $\T M\to M$ with Dolbeault 1-derivation $\mathcalbb{D}^{r,g}$ and linear complex structure on the total space $\T M$ given by $K_{r,g}$. Moreover, for any holomorphic 3-form $H$, $\mathcalbb{D}^{r,g}$ is a Courant 1-derivation with respect to the $H$-twisted Courant bracket on $\T M$.
\end{cor}

The last assertion follows from Prop.~\ref{prop:Dg}.
 
For a pseudo-K\"ahler structure $(r,g)$, we will see below that there is an explicit isomorphism relating the holomorphic structures on $\T M\to M$ defined by $\mathcalbb{D}^r$ and $\mathcalbb{D}^{r,g}$.

%%%%%%%%%%%%%%%%%%%%%%%%%%%%%%%%%%%%%%%%%%%%%%%%%%%%%%%%
\subsection{B-field transformations}
Given a closed 2-form $B\in \Omega^2(M)$, the map 
$$
\tau_B: \T M \to \T M, \qquad \tau_B(X, \alpha) = (X, \alpha + i_XB),
$$
is a Courant automorphism of $\T M$ for any $H$-twisted Courant bracket (if $B$ is not closed, then $\tau_B$ intertwines Courant brackets twisted by $H$ and $(H-dB)$). Such $\tau_B$ is called a {\em gauge transformation} \cite{severaweinstein}, or a {\em $B$-field transform} \cite{gualt11}. 

% Let $\omega \in \Omega^2(M)$ be a closed 2-form. The map $\tau_\omega: \T M \to \T M$ defined by $\tau_\omega(X,\alpha)=(X,\alpha+i_X\omega)$ is a automorphism of the Courant algebroid. 

We say that two Courant 1-derivations $\calD_1= (D_1, l_1, r_1)$ and $\calD_2= (D_2, l_2,r_2)$ on $\T M$ are \textit{gauge equivalent} if $r_1=r_2$ and there exists a closed 2-form $B$ such that 
$$
D_2 = \tau_B \circ D_1 \circ \tau_B^{-1} \text{ and } l_2 = \tau_B \circ l_1 \circ \tau^{-1}_B.
$$
The  Nijenhuis and almost complex equations \eqref{eq:van_Nijenhuis} and \eqref{eq:square_minus_id}, respectively, are invariant by gauge equivalence.

Consider a Courant 1-derivation $\calD$ determined by $r$, $g$ and $\Sigma$ via \eqref{dfn:std_der}; {\em here we no longer assume that the symmetric bilinear form $g$ is nondegenerate.} Then $\calD$ is gauge equivalent to $\mathcalbb{D}^r$ if and only if there exists a closed 2-form $B$ such that
\begin{equation}\label{eq:conjug}
g^\flat = B^\flat \circ r - r^*\circ B^\flat, \qquad  \Sigma_X = B^\flat \circ D^{r}_X - D^{r,*}_X \circ B^\flat.
\end{equation}
In particular, $\tau_B$ preserves the 1-derivation $\mathcalbb{D}^r$ (i.e., $g=0$ and $\Sigma=0$)  if and only if $B$ and $r$ are compatible in the sense of Magri-Morosi, see \cite[$\S$~5.2]{BDN} (cf. Example~\ref{ex:Dr}). 

Suppose that $r$ is a complex structure. 
Since $\mathcalbb{D}^r$ satisfies the almost complex equations \eqref{eq:square_minus_id}, any Courant 1-derivation $\calD$ gauge equivalent to it must satisfy these equations as well. In this case, if $\calD$ is determined by $r$, $g$ and $\Sigma$ (as in \eqref{dfn:std_der}), then necessarily $g^\flat \circ r = -r^*\circ g^\flat$ (see the proof of Proposition~ \ref{prop:almostcomplex}). So the pair $(r,g)$ gives rise to a 2-form $\omega \in \Omega^2(M)$ such that $\omega^\flat = g^\flat \circ r$.

We will now describe a natural class of Courant 1-derivations on $\T M$ that are gauge equivalent to $\mathcalbb{D}^r$. 

\begin{prop}\label{prop:closed_hol}
Suppose that $\calD$ is a Courant-Nijenhuis 1-derivation on $\T M$ determined by $r$, $g$ and $\Sigma$ (as in \eqref{dfn:std_der}) satisfying the almost complex equations. If the 2-form $\omega$ defined by $\omega^\flat = g^\flat \circ r$ is closed, then $\calD$ is gauge equivalent to $\mathcalbb{D}^r$.
\end{prop}

\begin{proof}
    Let $B=\frac{1}{2}\omega$, and let $\calD^B$ be the Courant-Nijenhuis 1-derivation  obtained by  conjugating $\calD$ with the $B$-field transform $\tau_B$. 
    Denote by  $r$, $g^B$ and $\Sigma^B$ the data that determine $\calD^B$. Then it is straightforward to check that
     $$
    g_B=0\;\;\; \mbox{ and } \;\;\; \Sigma^B = \Sigma + (B^\flat(D^r) - D^{r,*}(B^\flat)).
    $$ 
    Using \eqref{eq:Sigma_prop1} and \eqref{eq:Sigma_prop2} for $\calD^B$, we see that $\Sigma^B$  defines an element $H \in \Gamma(\wedge^2 T^*M \otimes T^*M)$ by 
    $$
    H(X,Y;Z) = \<\Sigma^B_X(Y),Z\>.
    $$
    As a consequence of \eqref{eq:im3}, $H$ is totally skew-symmetric, i.e. $H \in \Omega^3(M)$\footnote{Although not needed in the proof, one can check that \eqref{eq:im4} is equivalent to $dH=0$. In fact, Courant 1-derivations $\calD$ determined by $r$, $g$ and $\Sigma$ on $\T M$ with $g=0$ are in bijective correspondence with closed 3-forms.}. It now follows directly from the almost complex equations and Nijenhuis equations that $H$ must satisfy
$$
H(r(X),Y,Z) = -H(X,r(Y),Z) \;\;\; \text{ and } \;\; H(r(X),Y,Z) = H(X,r(Y),Z).
$$
Hence $H=0$, which means that $\Sigma^B=0$. Therefore $\calD^B = \mathcalbb{D}^r$, as we wanted to prove.
\end{proof}

\begin{cor}\label{cor:Bfield}
    For a pseudo-K\"ahler structure $(r,g)$, the Courant 1-derivations $\mathcalbb{D}^r$ and $\mathcalbb{D}^{r,g}$ are gauge equivalent.  
\end{cor}

It follows that the usual holomorphic structure on $\T M \to M$ (defined by $\mathcalbb{D}^r$, see Example~\ref{ex:Dr}) and the one modified by the metric $g$ (defined by $\mathcalbb{D}^{r,g}$, see Cor.~\ref{cor:Krg})  are isomorphic through a gauge transformation $\tau_B$. Since $\tau_B$ 
establishes a bijective correspondence between the sets of $\mathcalbb{D}^r$-invariant and $\mathcalbb{D}^{r,g}$-invariant lagrangian subbundles, it is clear that a lagrangian subbundle $L\subset \T M$ is holomorphic with respect to the holomorphic structure modified by $g$  if and only if $\tau_B(L)$ is holomorphic in the usual sense. (In particular, this shows that $\mathcalbb{D}^{r,g}$-invariant Dirac structures are much less restrictive than those for $\mathcalbb{D}^g$, c.f. Example \ref{ex:dirac_g}).

%$\mathcalbb{D}^r$-invariant (iholomorphic with respect to the complex structure $r$). 

%If $\tau_B$ is a gauge transformation relating Courant 1-derivations, then it .
%For gauge equivalent Courant 1-derivations, their sets of invariant lagrangian subbundles are in bijective correspondence under the gauge transformation. 
%Hence, for a pseudo-K\"ahler structure $(r,g)$ and  $B=\frac{1}{2}\omega$ (with $\omega$ defined by $\omega^\flat = g^\flat \circ r$), 
%, i.e., holomorphic. In particular, the graph of a 2-form $L=\mathrm{graph}(\Omega)$ is $\mathcalbb{D}^{r,g}$-invariant if and only if $\Omega+\frac{1}{2}\omega$ is a holomorphic 2-form (c.f. Example \ref{ex:dirac_g}). \textcolor{blue}{I will edit a bit more, not sure the point is clear}
%\end{rmk}

\begin{comment}
\begin{rmk}\label{rem:g=0}
{\em \textcolor{blue}{if $r g = -g r$, one can use B-field to gauge way $g$ in the first equation of \eqref{dfn:std_der}, so that $l$ is unchanged... so with this condition all Courant 1-derivations are classified by $\Sigma$, e.g. in holomorphic case}}
\hfill $\diamond$
\end{rmk}
\end{comment}
%\textcolor{blue}{compare with holomorphic Courant associated with GKS, do we have special cases? }

%%%%%%%%%%%%%%%%%%%%%%%%%%%%%%%%%%%%%%%%%%%%%%%%%%%%%%%%%%%%%%%%%%%

\section{Holomorphic Courant algebroids}\label{sec:hol}
In this section we show that holomorphic Courant algebroids can be seen as special cases of  Courant-Nijenhuis structures.

For a complex manifold $(M,r)$ and holomorphic vector bundle $\mathcal{E}\to M$, we  denote by $\mathcal{O}$ the  sheaf of holomorphic functions on $M$ and by $\E$ the sheaf of holomorphic sections of $E$. 
%We denote by $\mathcal{O}$ the sheaf of holomorphic functions on $M$.

\begin{defn} A holomorphic Courant algebroid is a holomorphic vector bundle $\mathcal{E} \to M$ endowed with a holomorphic non-degenerate symmetric $\mathcal{O}$-bilinear pairing $\<\cdot,\cdot\>: \E \times \E \to \mathcal{O}$, a holomorphic vector bundle map $\Panchor: E \to T^{1,0}M$ and a $\C$-bilinear bracket $\Cour{\cdot,\cdot}: \E \times \E \to \E$ satisfying axioms \eqref{C1}, \dots, \eqref{C5} in Definition \ref{dfn:courant_alg} with $\Panchor$ in place of $\an$, $f$ a local holomorphic function, and $\sigma_1, \, \sigma_2$ and $\sigma_3$ local holomorphic sections.
\begin{comment}
\begin{align}
\label{HC2}\tag{HC2} \Panchor(a_1)\llbracket a_2, a_3\rrbracket & = \llbracket \Cour{a_1,a_2}_{\mathrm{hol}}, a_3\rrbracket + \llbracket a_2, \Cour{a_1, a_3}_{\mathrm{hol}}\rrbracket
\\
\label{HC3}\tag{HC3} \Cour{a_1,a_2} + \Cour{a_2,a_1} & = \Panchor^*(\partial\llbracket a_1, a_2\rrbracket)
\end{align}
where $\Panchor^*: (T^{1,0}M)^* \rightarrow E^* \simeq E$ is the map dual to the anchor $\Panchor$ where we identify $E^* \simeq E$ via the non-degenerate pairing.
\end{comment}
\end{defn}

 For the sake of completeness, we start by recalling that there exists a natural (real) smooth Courant algebroid underlying a holomorphic one (see e.g. \cite{gualt14}).  For a holomorphic vector bundle $\mathcal{E}$, we denote its underlying real, smooth vector bundle by $E$.

\begin{prop}\label{prop:smooth_real_cour}
Let $(\mathcal{E}, \Panchor, \Cour{\cdot,\cdot}, \<\cdot, \cdot\>)$ be a holomorphic Courant algebroid. There exists a unique (real) smooth Courant algebroid structure $(\an,  \<\cdot, \cdot\>_{\Cinfty}, \Cour{\cdot,\cdot}_\Cinfty)$ on $E$ such that
\begin{itemize}
\item[(1)] $\Panchor(\sigma) = \frac{1}{2}(\an(\sigma) - \ii r(\an(\sigma)))$
\item[(2)] $\Cour{\sigma_1,\sigma_2}_\Cinfty = \Cour{\sigma_1,\sigma_2}$
\item[(3)] $\<\sigma_1,\sigma_2\>_\Cinfty = \Real(\< \sigma_1,\sigma_2\>)$
\end{itemize}
for all local holomorphic sections $\sigma, \sigma_1, \sigma_2 \in \E(U)$.
\end{prop}

\begin{proof}
We outline the proof, split in two parts:  uniqueness and existence.
\medskip

\paragraph{\bf Uniqueness} We will show that the restriction of the real smooth Courant algebroid structure (satisfying (1), (2) and (3)) to holomorphic sections  is sufficient to completely determine it. Let $U \subset M$ be an open subset for which there exists a frame of holomorphic sections $\{\sigma_1, \dots, \sigma_n\} \subset \E(U)$. Any smooth section $\sigma$ of $E$  over $U$ can be expressed uniquely as 
\begin{equation}\label{eq:holcomb}
\sigma = \sum_{k=1}^n f_k \sigma_k + g_k \,l(\sigma_k), 
\end{equation}
for $f_k, \, g_k \in C^{\infty}(U)$, where $l: E\to E$ the fibrewise complex structure on $E$. From the $C^\infty(U)$-linearity of $\an$, we see that $\an$ is completely characterized by $\mathfrak{p}(\sigma_k)$ (note that the $\C$-linearity of $\mathfrak{p}$ implies that $\an \circ l = r \circ \an)$. Also, the $C^{\infty}(U)$-bilinearity of $\<\cdot, \cdot\>_\Cinfty$ implies that it is determined by $\<\sigma_j, \sigma_k\>$, noticing that 
\begin{equation}\label{eq:l*}
\<\sigma_j, l(\sigma_k)\>_{\Cinfty} = -\mathrm{Im}(\<\sigma_j,\sigma_k\>).
\end{equation}
Finally, the Leibniz equation \eqref{C3} implies that $\Cour{\cdot,\cdot}_\Cinfty$ is completely determined by $\Cour{\sigma_j,\sigma_k}$ (note that $\Cour{\sigma_j, l(\sigma_k)}_\Cinfty= l(\Cour{\sigma_j,\sigma_k})$).  
\medskip

\paragraph{\bf Existence} By uniqueness, it suffices to describe the real smooth Courant algebroid structure locally. We will use the isomorphism of $C^\infty(U)$-modules
$$
C^{\infty}(U,\C) \otimes_{\O(U)} \E(U) \ni (f+\ii g) \otimes \sigma \;  \mapsto \; f \sigma + g \, l(\sigma)  \in \Gamma(U,E)
$$
to construct the Courant algebroid structure locally.  The anchor and the bracket are given by
$$
\an(\psi \otimes \sigma) = \mathrm{Re}(\psi) \an(\sigma) + \mathrm{Im}(\psi) r(\an(\sigma)), \,\,\,\,\, \<\psi_1\otimes \sigma_1, \psi_2 \otimes \sigma_2\>_\Cinfty = \mathrm{Re}(\psi_1 \psi_2 \<\sigma_1, \sigma_2\>).
$$
Note that $\<\cdot,\cdot\>_\Cinfty$ is non-degenerate and, for $h \in \O(U)$,
\begin{equation}\label{eq:anchor_on_hol_functions}
\Lie_{\an(\psi \otimes \sigma)} h = \psi \, \Lie_{\mathfrak{p}(\sigma)}h.
\end{equation}
We can now define $\an^*: \Omega^1(U) \to \newcourant$ naturally as 
\begin{equation*}
\<\an^*(\alpha), \psi\otimes \sigma\>_\Cinfty = i_{\an(\psi\otimes \sigma)}\,\alpha.
\end{equation*}
The Courant bracket on $\newcourant$ is given by 
\begin{align}
\label{dfn:smooth_cour}\Cour{\psi_1 \otimes \sigma_1, \psi_2 \otimes \sigma_2}_\Cinfty = & \psi_1 \psi_2 \otimes \Cour{\sigma_1,\sigma_2} + \Lie_{\an(\psi_1\otimes \sigma_1)}(\psi_2) \otimes \sigma_2 - \Lie_{\an(\psi_2\otimes \sigma_2)}(\psi_1) \otimes \sigma_1
\\
\nonumber & + \Re(\psi_2 \<\sigma_1, \sigma_2\>) \an^* d\Re(\psi_1)- \Im(\psi_2\<\sigma_1,\sigma_2\>) \an^* d\Im(\psi_1).
\end{align}
One can check that $\an$, $\<\cdot,\cdot\>_\Cinfty$ and $\Cour{\cdot,\cdot}_\Cinfty$ are well-defined in the sense that their definitions agree on $(h\psi) \otimes \sigma$ and $\psi\otimes (h\sigma)$, for $h \in \O(U)$. The axioms \eqref{C1}--\eqref{C5} follow by direct inspection. 
%We refer to Appendix \ref{app:hol_courant} for the details.
\end{proof}

%We will  denote the smooth Courant algebroid associated with $(E, \Panchor, \Cour{\cdot,\cdot}, \<\cdot, \cdot\>)$ by   $E_\Cinfty$.

Following Examples~\ref{ex:hol} and \ref{ex:holrev}, we regard a holomorphic vector bundle $\mathcal{E}$ as a pair $(E, \mathcal{D}^{\mathrm{Dolb}})$, where
$E\to M$ is a real, smooth vector bundle endowed with a Dolbeault
1-derivation $\mathcal{D}^{\mathrm{Dolb}}$ defining its holomorphic structure. 
Our main result in this section is an equivalence between holomorphic Courant algebroid structures on $\mathcal{E}$ and real Courant algebroid structures on $E$ for which $\mathcal{D}^{\mathrm{Dolb}}$ is a Courant 1-derivation.

%that holomorphic Courant algebroids are special types of Courant-Nijenhuis algebroids: we have an equivalence between holomorphic Courant algebroid structures on a holomorphic vector bundle $(E, (D^{\mathrm{hol}}, l, r))$ and real Courant algebroid structures on $E$ for which $(D^\mathrm{hol}, l, r)$ is a Courant 1-derivation.

%\textcolor{blue}{make notation clearer}

\begin{thm}\label{thm:holcourant} Consider a holomorphic vector bundle $\mathcal{E} = (E, \mathcal{D}^{\mathrm{Dolb}})$.
\begin{itemize}
\item[(a)] If $(\Panchor,  \<\cdot,\cdot\>, \Cour{\cdot,\cdot})$ is a holomorphic Courant algebroid structure on $\mathcal{E}$, then $\mathcal{D}^{\mathrm{Dolb}}$ is a Courant 1-derivation of the underlying real Courant algebroid structure on $E$ (so $\mathcal{D}^{\mathrm{Dolb}}$ makes $E$ into a Courant-Nijenhuis algebroid). 

\item[(b)] If $(\an,  \<\cdot, \cdot\>, \Cour{\cdot, \cdot})$ is a Courant algebroid structure on $E$ such that  
$\mathcal{D}^{\mathrm{Dolb}}$ is a Courant 1-derivation,
then $\Cour{\cdot, \cdot}$ restricts to a $\C$-bilinear bracket %$\Cour{\cdot, \cdot}_{\mathrm{hol}}$ 
on $\E$ in such a way that it defines a holomorphic Courant algebroid structure on $\mathcal{E}$ together with
$$
\Panchor(\sigma) = \frac{1}{2}(\an(\sigma) - \ii r(\an(\sigma))), \,\,\,\,\,\, 
\< \sigma_1, \sigma_2 \>_{\mathrm{hol}} = \<\sigma_1, \sigma_2\> - \ii \<\sigma_1, l(\sigma_2)\>.
$$
\end{itemize}
\end{thm}

The constructions in (a) and (b) are inverses of one another.

\begin{proof}

%By \eqref{eq:Dhol}, we see that a section $\sigma \in \Gamma(U, E)$ is holomorphic if and only if $D_X^{\mathrm{hol}}(\sigma)=0$ for all $X\in \mathfrak{X}(M)$.

To prove part (a), we must show that $\mathcal{D}^{\mathrm{Dolb}} = (D, l,r)$ is a Courant 1-derivation, i.e., that it satisfies  (CN1)--(CN4) in Def.~\ref{def:Cder} with respect to the real Courant algebroid $E$. The $\C$-linearity of $\Panchor$ implies that $\an \circ l = r \circ \an$, so (CN1) holds. To verify conditions (CN2), (CN3) and (CN4), we must check,  for each open subset $U \subset M$, the vanishing of   the following expressions:
\begin{flalign*}
&W_2(X,\sigma_1)  =   \an (D_X(\sigma_1)) - D_X^r(\an(\sigma_1)),\\
&W_3(\sigma_1,\sigma_2) = l(\Cour{ \sigma_1, \sigma_2 }_\Cinfty)- \left( \Cour{ \sigma_1,l(\sigma_2) }_\Cinfty - D_{\an(\sigma_2)}(\sigma_1) - \an^*( C(\sigma_1,\sigma_2)) \,\right), \\
&W_4(X,\sigma_1,\sigma_2)  =   D_X(\Cour{ \sigma_1,\sigma_2 }_\Cinfty) - \left(\Cour{ \sigma_1,D_X(\sigma_2) }_\Cinfty - \Cour{ \sigma_2,D_X(\sigma_1)}_\Cinfty + D_{[\an(\sigma_2),X]}(\sigma_1)\right.\\
&\qquad\qquad\qquad \quad \left. - D_{[\an(\sigma_1),X]}(\sigma_2) - \an^*(i_X\,dC(\sigma_1,\sigma_2))\,\right),
\end{flalign*}
for $X \in \frakx(U)$ and $\sigma_1, \sigma_2$ smooth sections of $E$ over $U$.

The vanishing of $W_2$ follows from the fact that, since  $\an$ is holomorphic, it must intertwine $\mathcal{D}^{\mathrm{Dolb}}$ and $\mathcal{D}^r$ (recalling that $\mathcal{D}^r$ is the Dolbeault 1-derivation encoding the holomorphic structure on $TM$, see Example~\ref{ex:hol}), so (CN2) holds.

 Using that $\an \circ l = r \circ \an$ and $l=l^*$ (by \eqref{eq:l*}), one can show that
 %that both $W_2$ and 
 $W_3$ is $C^\infty(U)$-linear in each entry, whereas $W_4$ is $C^ \infty(U)$-linear in the first entry and satisfies
 \begin{align*}
W_4(X, \sigma_1, f\sigma_2) - fW_4(X,\sigma_1, \sigma_2) & =  (\Lie_X f)  \,W_3(\sigma_1,\sigma_2) - (\Lie_{W_2(X,\sigma_1)}f)\, \sigma_2 \\
& = (\Lie_X f)  \,W_3(\sigma_1,\sigma_2), \\
%** W_4(X, f \sigma_1, \sigma_2) - fW_4(X,\sigma_1, \sigma_2) = & \Lie_X \<\sigma_1,\sigma_2\>_\Cinfty \left(l(\an^*df) - \an^*(r^*df)\right) - (\Lie_X f) W_3 (\sigma_2, \sigma_1) \\
%& \hspace{-12pt} + \<\sigma_1,\sigma_2\>_\Cinfty \left( D_X^\hol(\an^*df) - \an^*(D^{r,*}_X(df)) \right) - W_1(X,\sigma_1,\sigma_2) \an^*df, \\
W_4(X, f \sigma_1, \sigma_2) - fW_4(X,\sigma_1, \sigma_2) & =   - (\Lie_X f) W_3 (\sigma_2, \sigma_1)  - W_1(X,\sigma_1,\sigma_2) \an^*df
\end{align*}
where 
$$
W_1(X,\sigma_1,\sigma_2)  =  \<D_X(\sigma_1),\sigma_2\>_\Cinfty + \<\sigma_1, D_X(\sigma_2)\>_\Cinfty - \Lie_X\<\sigma_1, l(\sigma_2)\>_\Cinfty + \Lie_{r(X)}\<\sigma_1,\sigma_2\>_\Cinfty.
$$
Note that $W_1$ is $C^\infty(U)$-linear in each entry due to \eqref{eq:Leibniz}.

We claim that $W_1$, $W_3$ and $W_4$ vanish on holomorphic sections. Indeed, for $W_3$ and $W_4$ this follows from the facts that $D_X$ is zero on holomorphic sections (by \eqref{eq:Dhol})
%$\an$ is holomorphic, $D^r$ is the 1-derivation associated with the holomorphic structure on $TM$ (see Example~\ref{ex:Dr} and \cite[Thm.~5.3]{Dru}), 
and that $\Cour{\cdot,\cdot}_\Cinfty$ restricts to the $\C$-bilinear  holomorphic Courant bracket on holomorphic sections. For $W_1$, we use the additional fact that
$$
\Lie_X\<\sigma_1, l(\sigma_2)\>_\Cinfty - \Lie_{r(X)}\<\sigma_1,\sigma_2\>_\Cinfty = -\Im(\Lie_{X+\ii r(X)}\<\sigma_1,\sigma_2\>)=0,
$$
for holomorphic $\sigma_1, \sigma_2 \in \E(U)$, since $\<\sigma_1,\sigma_2\>\in \O(U)$.

Since smooth sections can be locally expressed by means of a frame of holomorphic sections as in \eqref{eq:holcomb}, the $C^\infty(U)$-multilinearity of $W_1$ and $W_3$ implies that they vanish. This in turn implies that $W_4$ is also multilinear over $C^\infty(U)$, and hence also vanishes.
This concludes the proof of (a). 

Let us prove part (b). Since $D$ vanishes on holomorphic sections, \eqref{eq:im4} implies that $\Cour{\E(U), \E(U)}\subset \E(U)$, and it follows from \eqref{eq:im3} that the restricted bracket on holomorphic sections is $\C$-bilinear. Regarding the anchor, \eqref{eq:im1} implies that $\Panchor$ is $\C$-linear, and \eqref{eq:im2} says that it is holomorphic. It remains to check that $\<\E(U),\E(U)\> \subset \O(U)$. This is a consequence of the fact that, for local holomorphic sections $\sigma_1, \sigma_2$, the duality equation  \eqref{defn:dual_derivation} implies that 
$$
\Lie_{X+\ii r(X)}\<\sigma_1,\sigma_2\>_\hol =  \<\dob_{X+\ii r(X)}\,\sigma_1, \sigma_2 \>_\hol +  \<\sigma_1, \dob_{X+\ii r(X)}\,\sigma_2\>_\hol =0,
$$
where $\dob$ is the flat $T^{0,1}$-connection  \eqref{eq:Dhol}. Therefore $\<\sigma_1,\sigma_2\>_\hol$ is holomorphic, thus concluding the proof.
\end{proof}

\begin{ex} {\em
Let $(M,r)$ be a complex manifold and $H$ a closed holomorphic 3-form.
Then $\T M$, viewed as a holomorphic vector bundle, carries an $H$-twisted holomorphic Courant algebroid structure (analogous to Example~\ref{ex:Htwisted}). From the perspective of Theorem~\ref{thm:holcourant}, this corresponds to the fact that $\mathcalbb{D}^r$ is a Courant 1-derivation of the $H$-twisted Courant bracket on $\T M$, viewed as a real vector bundle (see Example~\ref{ex:Dr}).  }
\hfill $\diamond$
\end{ex}

\begin{rmk}{\em
Any holomorphic Courant algebroid whose  underlying real Courant algebroid is $\T M$ with the $H$-twisted Courant bracket, for a closed 3-form $H$, is completely characterized by a Courant-Nijenhuis 1-derivation $\calD$ on $(\T M, H)$ determined by the data $r$, $g$ and $\Sigma$ (as in \eqref{dfn:std_der}), where $r$ is the complex structure on $M$. Then $g$ and $\Sigma$ must satisfy the equations corresponding to the fact that $\calD$  is Courant-Nijenhuis and almost complex. If the 2-form $\omega^\flat = g^\flat \circ r$ is closed, Proposition \ref{prop:closed_hol} implies that the holomorphic Courant algebroid determined by $\calD$ is isomorphic to the one determined by $\mathcalbb{D}^r$. In general this is not the case, but one can always use $\omega$ to gauge away $g$ leaving the construction of more general holomorphic Courant algebroids as a problem of choosing $\Sigma$ suitably; this provides a different approach to the classification of holomorphic Courant algebroids in \cite[Prop.1.3]{gualt14}. }
\hfill $\diamond$
\end{rmk}

%\textcolor{blue}{mention: compatible dirac are holomorphic dirac; mention examples of hol courant algebroids in GKG, classification of exact ones...}

%%%%%%%%%%%%%%%%%%%%%%%%%%%%%%%%%%%%%%%%%%%%%%%%%%%%%%%%%%%%%%%%%%%%
\section{Lagrangian splittings and doubles}\label{sec:double}

\subsection{Lagrangian splittings of Courant algebroids}\label{subsec:proto}
Let $(E,\<\cdot,\cdot\>, \an, \Cour{\cdot,\cdot} )$ be a Courant algebroid. By a {\em lagrangian splitting} of $E$ we mean a decomposition 
$E = A \oplus B$, where $A$ and $B$ are lagrangian subbundles. In this case, there is an isomorphism $B\cong A^*$ via the pairing that yields an identification $E=A\oplus A^*$
as pseudo-euclidean vector bundles, where $A\oplus A^*$ is equipped with its canonical pairing
$$
\<(a,\alpha),(b,\beta)\>:= \beta(a) + \alpha(b).
$$

Let $p_A: E\to A$ and $p_{A^*}: E\to A^*$ the natural projections onto $A$ and $A^*$, respectively.
The anchor and bracket on $E$ induce the following structures on $A$ and $A^*$: an anchor 
$\rho:= \an|_A: A\to TM$, along with an $\R$-bilinear bracket $[\cdot,\cdot]$ on $\Gamma(A)$ 
and an element $\varphi \in \Gamma(\wedge^3 A^*)$
given by
$$
[a,b] := p_A(\Cour{(a,0),(b,0)}), \qquad
i_bi_a\varphi := p_{A^*}(\Cour{(a,0), (b,0)}),
$$
and, similarly, an anchor $\rho_*: A^*\to TM$, bracket $[\cdot,\cdot]_*$ on $\Gamma(A^*)$ and element $\chi \in \Gamma(\wedge^3 A)$. Note that $A$ and $A^*$, endowed with their anchors and brackets, become pre-Lie algebroids (see $\S$ \ref{subsec:IM}). We denote the corresponding operators by 
$$
d_A: \Gamma(\wedge^\bullet A^*)\to \Gamma(\wedge^{\bullet + 1}A^*), \qquad d_{A^*}: \Gamma(\wedge^\bullet A)\to \Gamma(\wedge^{\bullet + 1}A).
$$

With respect to these structures, the Courant bracket on $E=A\oplus A^*$ is given by
\begin{equation}\label{eq:double_bracket}
    \Cour{(a,\alpha),(b, \beta)}= ( [a,b] + \mathcal{L}_{\alpha}b- i_\beta\, d_{A^*}a + i_\beta i_\alpha \chi, \, [\alpha,\beta]_* + \mathcal{L}_a\beta-i_b \, d_A\alpha + i_b i_a \varphi),
\end{equation}
where $\mathcal{L}_\alpha=i_\alpha d_{A^*} + d_{A^*} i_\alpha$, similarly for $\mathcal{L}_a$.

Let $(A, \rho, [\cdot,\cdot])$ and $(A^*, \rho_*, [\cdot,\cdot]_*)$ be pre-Lie algebroids in duality, further equipped with sections $\varphi \in \Gamma(\wedge^3 A^*)$ and $\chi \in \Gamma(\wedge^3 A)$.
The pair $(A,A^*)$ is called a {\em proto bialgebroid}  when the anchors, brackets and 3-sections satisfy compatibility conditions (spelled out in \cite{roy:quasi}, see also \cite{yvette:quasitwisted}) saying that the bracket \eqref{eq:double_bracket} on $\Gamma(A\oplus A^*)$ makes $A\oplus A^*$ into a Courant algebroid with anchor $\an = \rho + \rho_*$ and canonical pairing, called the {\em double} of $(A,A^*)$.

We therefore obtain the following equivalence: any Courant algebroid equipped with a lagrangian splitting yields a proto bialgebroid, and any proto bialgebroid gives rise, by means of its double, to a Courant algebroid endowed with a lagrangian splitting.

The following are special cases of interest of this correspondence.

\begin{itemize}
\item When $\varphi=0$ and $\chi=0$, a proto bialgebroid $(A,A^*)$ is a Lie bialgebroid, i.e.,  $(A, A^*)$ is a pair of Lie algebroids $(A, \rho, [\cdot, \cdot]) $ and $(A^*,\rho_*, [\cdot, \cdot]_{*})$ in duality such that the Lie-algebroid differential $d_{A^*}$ and the Schouten bracket $[\cdot,\cdot]$ on $\Gamma(\wedge A)$ satisfy
\begin{equation}\label{eq:bialg}
d_{A^*}[a,b] = [d_{A^*}a, b] + [a, d_{A^*}b], \qquad \forall \, a,  b \in \Gamma(A).
\end{equation}
Lie bialgebroids are in correspondence with Courant algebroids equipped with a lagrangian splitting by Dirac structures  \cite{LiuXuWe}, known as {\em Manin triples}.

\item When $\chi=0$, a proto bialgebroid $(A,A^*)$ is a Lie quasi-bialgebroid \cite{roy:quasi}, in which case $A$ is a Lie algebroid, $d_{A^*}$ satisfies \eqref{eq:bialg}, $d_{A^*}^2=[\chi,\cdot]$, and $d_{A^*}\chi=0$. Lie quasi-bialgebroids correspond to Courant algebroids equipped with a splitting given by a Dirac structure and a lagrangian complement, known as {\em Manin quasi-triples}. 
\end{itemize}

For a Courant algebroid $E\to M$, a Lagrangian splitting
$E=A\oplus A^*$ induces a bivector field $\pi$ on $M$ via
\begin{equation}\label{eq:pi}
\pi^\sharp = \rho_* \circ \rho^*: T^*M \to TM
\end{equation}
that satisfies 
$$
\frac{1}{2}[\pi,\pi]= \rho(\chi)+ \rho_*(\varphi),
$$
see \cite[$\S$ 3.2 and 3.4]{liblandmein}. In particular $\pi$ is a Poisson structure when $(A,A^*)$ is a Lie bialgebroid.

%%%%%%%%%%%%%%%%%%%%%%%%%%%%%%%%%%%%%%%%%%%%%%%%%%%%%%%%%%%%%%%%%%%%%%
\subsection{Lagrangian splittings of Courant 1-derivations}

Consider a Courant algebroid $E$ equipped with a lagrangian splitting, that we write as $E= A \oplus A^*$. 

\bigskip

\noindent{\bf Assumption}. {\em We assume throughout this subsection that
 $$
 \mathcalbb{D} = (\D, \ell, r)
 $$ 
 is a 1-derivation on the vector bundle $E$ that keeps the splitting invariant, i.e., the subbundles $A$ and $A^*$ are $\mathcalbb{D}$-invariant (as in Def.~\ref{def:compat}). We will further assume that
 $\mathcalbb{D}$ is symmetric ($\mathcalbb{D}=\mathcalbb{D}^*$), in which case 
 the restricted 1-derivations on $A$ and $A^*$ are dual to one another.}
 
 \bigskip
 
 We denote the 1-derivation on $A$ by $\mathcal{D}=(D, l, r)$, so that the 1-derivation on $A^*$ is $\mathcal{D}^*= (D^*, l^*, r)$ and
$$
\D = (D, D^*), \qquad \ell= (l,l^*).
$$

We keep the notation from $\S$ \ref{subsec:proto} for the pre-Lie algebroids $(A, \rho, [\cdot,\cdot])$ and $(A, \rho_*, [\cdot,\cdot]_*)$, with 3-sections $\chi \in \Gamma(\wedge^3 A)$ and $\varphi \in \Gamma(\wedge^3 A^*)$.

\begin{thm}\label{thm:manin}
The 1-derivation $\mathcalbb{D} = (\D, \ell, r)$ is a Courant 1-derivation of $E=A\oplus A^*$ if and only if the following conditions hold:
\begin{itemize}
    \item $\mathcal{D}$ is compatible with the pre-Lie algebroid $(A, \rho, [\cdot,\cdot])$, and
    $$
    \varphi \in \Gamma_{l}(\wedge^3 A^*), \qquad D^*(\varphi) =0;
    $$
      
    \item $\mathcal{D}^*$ is compatible with the pre-Lie algebroid $(A^*, \rho_*, [\cdot,\cdot]_*)$, and
  $$
    \chi \in \Gamma_{l^*}(\wedge^3 A), \qquad D (\chi) =0.
    $$  
\end{itemize}
Moreover, $\mathcalbb{D}$ is Nijenhuis (resp. Dolbeault)  if and only if so is $\mathcal{D}$.
\end{thm}

%$\mathfrak{D}$ $\mathcal D$ $\mathcalbb{D}$ $\mathscr{D}$ 

\begin{proof}
We must show the equivalence between conditions (CN1)--(CN4) (in Definition~\ref{def:Cder}) for $\mathcalbb{D}$ and (IM1)--(IM4) (in Definition~\ref{def:LAcomp}) for both $\mathcal{D}$ and $\mathcal{D}^*$ along with the conditions on $\varphi$ and $\chi$ in the statement.

It directly follows from $\an= \rho+ \rho_*$ and $\ell = (l, l^*)$ that (CN1) for $\mathcalbb{D}$ is equivalent to (IM1) for $\mathcal{D}$ and $\mathcal{D}^*$. Similarly, the fact that $\D = (D, D^*)$ implies that (CN2) for $\mathcalbb{D}$ is equivalent to (IM2) for $\mathcal{D}$ and $\mathcal{D}^*$.

\medskip

\noindent{\em Claim 1. Assume that \eqref{eq:im1_lie} holds for $\mathcal{D}$ and $\mathcal{D}^*$. Then condition (CN3) for $\mathcalbb{D}$ is equivalent to (IM3) for $\mathcal{D}$ and $\mathcal{D}^*$, as well as $\varphi \in \Gamma_{l}(\wedge^3 A^*)$ and $\chi \in \Gamma_{l^*}(\wedge^3 A)$.}

\smallskip

Let us verify the claim. For sections of type $\sigma_1=(a,0)$, $\sigma_2=(b,0)$, (CN3) becomes
$$
\ell ([a,b] + i_bi_a\varphi) = [a, l(b)] + i_{l(b)}i_a\varphi - D_{\rho(b)}(a),
$$
which splits into
$$
l([a,b]) = [a, l(b)]  - D_{\rho(b)}(a) \;\; \mbox{ and } \;\; l^*(i_bi_a\varphi) = i_{l(b)}i_a\varphi.
$$
These conditions hold for all $a, b\in \Gamma(A)$ if and only if  (IM3) holds for $\mathcal{D}$ and $\varphi \in \Gamma_{l}(\wedge^3 A^*)$. Similarly, (CN3) for sections of type $\sigma_1=(0,\alpha)$, $\sigma_2=(0,\beta)$ is equivalent to (IM3) for $\mathcal{D}^*$ and $\chi \in \Gamma_{l^*}(\wedge^3 A)$. For sections of type $\sigma_1=(0,\alpha)$, $\sigma_2=(b,0)$, (CN3) amounts to two equations:
\begin{align}
l^*(i_b d_A \alpha) &= i_{l(b)}d_A\alpha + D^*_{\rho(b)}(\alpha) + \< D^*_{\rho(\cdot)}(\alpha), b\>, \label{eq:CN3mix1}  \\
  l(\Lie_\alpha b) &= \Lie_\alpha l(b) - \< D^*_{\rho_*(\cdot)}(\alpha), b \>. \label{eq:CN3mix2} 
\end{align}
We will see that \eqref{eq:CN3mix1} (resp. \eqref{eq:CN3mix2}) follows directly from \eqref{eq:dual_im_a} for $\calD$ (resp. $\calD^*$) and $m=1$, which is equivalent to \eqref{eq:im3_lie} under the assumption that \eqref{eq:im1_lie} holds.
Indeed, note that \eqref{eq:dual_im_a} for $m=1$ has the following alternative formulations
\begin{equation}\label{eq:dual_im_alternative}
\<D^*_{\rho(\cdot)}(\alpha), b\> = l^*(i_bd_A\alpha) - i_bd_A(l^*(\alpha)) \,\,\, \text{(similarly,  } \<\alpha, D_{\rho_*(\cdot)}(b)\> = l(i_\alpha d_{A^*}b) - i_\alpha d_A(l(b)).)
\end{equation}
So \eqref{eq:CN3mix1} is obtained from adding up \eqref{eq:dual_im_a} and \eqref{eq:dual_im_alternative}. The second equation \eqref{eq:CN3mix2} follows from the Cartan formula $\Lie_\alpha = i_\alpha d_{A^*}+ d_{A^*} i_\alpha$ together with \eqref{defn:dual_derivation} and \eqref{eq:dual_im_alternative}. The equations corresponding to (CN3) for sections of type $\sigma_1=(a,0)$, $\sigma_2=(0,\beta)$ are entirely analogous to \eqref{eq:CN3mix1} and \eqref{eq:CN3mix2}, and hold for similar reasons. This proves the claim.

\medskip

\noindent{\em Claim 2. Assume that \eqref{eq:im1_lie}, \eqref{eq:im2_lie} and \eqref{eq:im3_lie} hold for $\mathcal{D}$ and $\calD^*$. Then ondition (CN4) for $\mathcalbb{D}$ is equivalent to (IM4) for $\mathcal{D}$ and $\mathcal{D}^*$, as well as $D^*(\varphi)=0$ and $D(\chi)=0$.}

\smallskip

To verify the claim, note that for sections of type $\sigma_1=(a,0)$, $\sigma_2=(b,0)$, (CN4) takes the form
\begin{align*}
 D_X([a,b]) + D_X^*(i_bi_a\varphi)  = &  [a,D_X(b)] + i_{D_X(b)}i_a\varphi - [b,D_X(a)] - i_{D_X(a)}i_b\varphi \\ 
 &+ D_{[\rho(b),X]}(a)
 - D_{[\rho(a),X]}(b).
\end{align*}
So in this case (CN4) amounts to (IM4) for $\mathcal{D}$ together with the following condition on $\varphi$ (using \eqref{defn:dual_derivation}):
$$
 \Lie_X\varphi(a,b,l(c))-\Lie_{r(X)}\varphi(a,b,c) - \varphi(D_Xa, b, c) -\varphi(a,D_Xb, c) - \varphi(a,b,D_Xc)=0,
$$
for all $a,b,c\in \Gamma(A)$. This last condition is the same as $D^*(\varphi) = 0$ when $\varphi \in \Gamma_l(\wedge^3 A^*)$ (see \eqref{eq:extension}). Similarly, (CN4) holds for sections of type $\sigma_1=(0,\alpha)$, $\sigma_2=(0,\beta)$ if and only if $\mathcal{D}^*$ satisfies (IM4) and $D(\chi)=0$.

For sections of type $\sigma_1=(0,\alpha)$, $\sigma_2=(b,0)$, (CN4) is equivalent to 
\begin{align}
 \label{eq:cn4_a} D_X(\Lie_\alpha b) & =  \Lie_\alpha(D_X b) + i_{D^*_X(\alpha)}d_{A^*}b - D_{[\rho_*(\alpha),X]}b - (\rho_*)^*i_X d\< D^*_{(\cdot)}\alpha, b\>,\\
 \label{eq:cn4_b} D_X^* (i_b d_A \alpha) &= i_{D_X(b)}d_A\alpha + \Lie_b D^*_X(\alpha) - D^*_{[\rho(b),X]}\alpha + \rho^*i_X d \, \<D^*_{(\cdot)}\alpha, b \>. 
\end{align}
These equations follows directly from \eqref{eq:dual_im_b} for $D$ and $D^*$ and $m=1$, which is equivalent to \eqref{eq:im4_lie} under the assumption that \eqref{eq:im1_lie}, \eqref{eq:im2_lie} and \eqref{eq:im3_lie} hold. Indeed, first notice that using \eqref{defn:dual_derivation} and $\rho_*^*D_X^{r,*}(d\<\alpha,b\>) = D_X(d_{A^*}\<\alpha,b\>)$ together with Cartan formula $\Lie_\alpha = i_\alpha d_{A^*} + d_{A^*} i_\alpha$, one can see that both equations are exactly the same under the change $\alpha \leftrightarrow b$. Now, using that
\begin{align*}
\<D_X^*(i_bd_A \alpha) - i_{D_X(b)}d_A\alpha, a\> & = -D_X^*(d_A\alpha)(a,b)\\
\<\Lie_b D^*_X(\alpha) + \rho^*i_X d \, \<D^*_{(\cdot)}\alpha, b \>, a\> & = -d_AD_X^*(\alpha)(a,b) + \Lie_X\<D_{\rho(a)}^*(\alpha),b\> + \<D^*_{[\rho(a),X]}(\alpha),b\>,
\end{align*}
one can check that \eqref{eq:cn4_b} is exactly \eqref{eq:dual_im_b} in degree 1.

The situation for sections of type $\sigma_1=(a,0)$, $\sigma_2=(0,\beta)$ is entirely analogous. This proves claim 2 and concludes the proof of the first part of the theorem.

The assertion about the Nijenhuis condition follows from the decompositions $\D=(D,D^*)$ and $\ell=(l,l^*)$, and the fact that $D$ is Nijenhuis (resp. Dolbeault) if and only is so is $D^*$, see \cite[Thm.~2.11]{Dru}.
\end{proof}

When $\mathcalbb{D}=(\D, \ell, r)$ is a Courant 1-derivation, there is also a compatibility with the bivector field $\pi \in \mathfrak{X}^2(M)$ in \eqref{eq:pi}, defined by the lagrangian splitting $E=A\oplus A^*$, see \cite[Prop.~4.6 (i)]{Dru}.

\begin{cor} 
The pair $(\pi, r)$ is compatible in the sense of \eqref{eq:compatpir}.  In particular, if the lagrangian splitting is by Dirac structures and $\mathcalbb{D}$ is Nijenhuis, then $(\pi,r)$ defines a Poisson-Nijenhuis structure.
\end{cor}

\begin{proof}
Using that $\pi^\sharp = \rho_* \circ \rho^*$, $\D=(D,D^*)$ and $\ell=(l,l^*)$, the first condition in \eqref{eq:compatpir} follows from condition (IM1) for $\mathcal{D}$ and $\mathcal{D}^*$, while the second condition in \eqref{eq:compatpir} follows from (IM2) for $\mathcal{D}$ and $\mathcal{D}^*$.
\end{proof}

Recall from \cite{Dru} that a {\em Lie-Nijenhuis bialgebroid} is a triple $(A, A^*, \mathcal{D})$, where $(A,A^*)$ is a Lie bialgebroid and $\mathcal{D}$ is a Nijenhuis 1-derivation on $A$ such that $\mathcal{D}$ is compatible with the Lie algebroid structure on $A$, and $\mathcal{D}^*$ is compatible with the Lie algebroid structure on $A^*$ (in the sense of Def.~\ref{def:LAcomp}). These are the infinitesimal objects corresponding to Poisson-Nijenhuis groupoids, see \cite[$\S$ 4.3]{Dru}.

By Theorem~\ref{thm:manin} we have the following enhancement of the known correspondence between Lie bialgebroids and Manin triples. 

\begin{cor}\label{cor:bialdouble}
Lie-Nijenhuis bialgebroids are equivalent to Courant-Nijenhuis algebroids equipped with a splitting by Dirac-Nijenhuis structures.
\end{cor}

When $\mathcalbb{D}$  is a Dolbeault 1-derivation, Theorem~\ref{thm:manin} gives the known correspondence between Lie (quasi-)bialgebroids and Manin (quasi-)triples in the holomorphic category.

\appendix

\end{document}